\documentclass[a4paper, 11pt, reqno]{amsart}
\usepackage[initials]{amsrefs}
\usepackage{amsmath,amscd,amssymb,latexsym, mathrsfs, amsthm}
\usepackage{hyperref}

\evensidemargin=\oddsidemargin
\DeclareMathVersion{can}
\DeclareMathAlphabet{\can}{OT1}{cmss}{m}{n}
\newtheorem{thm}{ Theorem}[section]
\newtheorem{cor}[thm]{Corollary}
\newtheorem{lem}[thm]{Lemma}
\newtheorem{prop}[thm]{Proposition}

\newtheorem{rem}[thm]{Remark}
\newtheorem{prob}[thm]{Problem}

\numberwithin{equation}{section}
\newcommand{\mc}{\mathcal}

\newcommand{\dis}{\displaystyle}

\newcommand{\A}{\mathbb{A}}

\newcommand{\OO}{\mathcal{O}}

\newcommand{\bF}{\mathbb{F}}
\newcommand{\CC}{\mathbb{C}}

\newcommand{\mb}{\mathbb}

\begin{document}
\title[Average Values of $L$--series for Real Characters in Function Fields]
{Average Values of $L$--series for Real Characters in Function Fields}%

\author[J.C. Andrade]{Julio C. Andrade}
\address{\rm  Department of Mathematics, University of Exeter, Exeter, EX4 4QF, United Kingdom}
\email{j.c.andrade@exeter.ac.uk}
\author[S. Bae]{Sunghan Bae}
\address{\rm Department of Mathematics, KAIST, Daejon 305-701, Korea}
\email{shbae@kaist.ac.kr}

\author[H. Jung]{Hwanyup Jung}
\address{\rm Department of Mathematics Education, Chungbuk National University, Cheongju  361-763, Korea}
\email{hyjung@chungbuk.ac.kr}
\subjclass[2010]{11G20 (Primary), 11M38, 11M50, 11R58, 14G10 (Secondary)}
\keywords{finite fields, function fields, hyperelliptic curves, $K$--groups, moments of quadratic Dirichlet $L$--functions, class number}


\begin{abstract}
We establish asymptotic formulae for the first and second moments of quadratic Dirichlet $L$--functions,
at the centre of the critical strip, associated to the real quadratic function field $k(\sqrt{P})$
and inert imaginary quadratic function field $k(\sqrt{\gamma P})$
with $P$ being a monic irreducible polynomial over a fixed finite field $\mathbb{F}_{q}$ of odd cardinality $q$
and $\gamma$ a generator of $\mathbb{F}_{q}^{\times}$.
We also study mean values for the class number and for the cardinality of the second $K$-group of maximal order
of the associated fields for ramified imaginary, real, and inert imaginary quadratic function fields over $\mathbb{F}_{q}$.

One of the main novelties of this paper is that we compute the second moment of quadratic Dirichlet $L$-functions associated
to monic irreducible polynomials. It is worth noting that the similar second moment over number fields is unknown.

The second innovation of this paper comes from the fact that, if the cardinality of the ground field is even then the task of average $L$-functions in function fields is much harder and, in this paper, we are able to handle this strenuous case and establish several mean values results of $L$-functions over function fields.
\end{abstract}

\maketitle
\tableofcontents

\section{Introduction and some basic facts}
\subsection{Introduction}
It is a profound problem in analytic number theory to understand the distribution of values of $L(s,\chi_{p})$,
the Dirichlet $L$--functions  associated to the quadratic character $\chi_{p}$, for fixed $s$ and variable $p$,
where for a prime number $p\equiv v \pmod 4$ with $v=1$ or $3$, the quadratic character $\chi_{p}(n)$ is defined
by the Legendre symbol $\chi_{p}(n)=(\frac{n}{p})$.
The problem about the distribution of values of Dirichlet $L$--functions with real characters $\chi$ modulo a prime $p$
was first studied by Elliott in \cite{Ell} and later some of his results were generalized by Stankus \cite{Stan}.

For $\mathrm{Re}(s)>\frac{1}{2}$, Stankus proved that $L(s,\chi_{p})$ is in a given Borel set $B$
in the complex plane with a certain probability which depends on $s$ and $B$.
However the same question is non-trivial if we consider $s$ in the center of the critical strip,
i.e., $\mathrm{Re}(s)=\frac{1}{2}$. In particular, it is a challenging (and open) problem to decide
if $L(\tfrac{1}{2},\chi_{p}) \ne 0$ for all quadratic characters $\chi_{p}$.
Appears that this fact was first conjectured by Chowla \cite{Chow}.

It is also a difficult problem to determine whether or not $L(\tfrac{1}{2},\chi_{p})\neq0$ for infinitely many primes $p$.
A natural strategy to attack this problem is to prove that $L(\tfrac{1}{2},\chi_{p})$ has a positive average value when $0<p\leq X$ and $X$ is large.
That is,
\begin{equation}
\sum_{\substack{p\leq X \\ p\equiv v \hspace{-0.7em}\pmod 4}} L(\tfrac{1}{2},\chi_{p}) > 0,
\end{equation}
when $X$ is large.
In this context, Goldfeld and Viola \cite{GV} have conjectured an asymptotic formula for
\begin{equation}
\sum_{\substack{p\leq X \\ p\equiv v \hspace{-0.7em}\pmod 4}}L(\tfrac{1}{2},\chi_{p}),
\end{equation}
and Jutila \cite{Jut} was able to establish the following asymptotic formula:

\begin{thm}[Jutila]\label{jutilat}
For $v=1$ or $3$, we have
\begin{equation}
\sum_{\substack{p\leq X \\ p\equiv v \hspace{-0.7em}\pmod 4}}(\log p)L(\tfrac{1}{2},\chi_{p})
=\frac{1}{4}X \left\{\log(X/\pi) + \frac{\Gamma'}{\Gamma}(v/4)+4\gamma-1\right\} + O\left(X(\log X)^{-A}\right),
\end{equation}
where the implied constant is not effectively calculable. The following estimate is effective:
\begin{equation}
\sum_{\substack{p\leq X \\ p\equiv v \hspace{-0.7em}\pmod 4}}(\log p)L(\tfrac{1}{2},\chi_{p})= \frac{1}{4}X\log X
+O\left(X(\log X)^{\varepsilon}\right).
\end{equation}
\end{thm}

In a recent paper \cite{AK13}, the authors raise the question about higher moments for the family of
quadratic Dirichlet $L$--functions associated to $\chi_{p}$. In other words, the problem is

\begin{prob}\label{mainprob}
Establish asymptotic formulas for
\begin{equation}\label{AKprob}
\sum_{\substack{p\leq X \\ p\equiv v \hspace{-0.7em}\pmod 4}}L(\tfrac{1}{2},\chi_{p})^{k},
\end{equation}
when $X\rightarrow\infty$ and $k>1$.
\end{prob}
The only known asymptotic formulae for \eqref{AKprob} are those given in the Theorem \ref{jutilat},
i.e., we have asymptotic formulas merely when $k=1$ and it is an important open problem for $k>1$.

The first aim of this paper is to study the function field analogue of the problem above
in the same spirit as those recent results obtained by Andrade and Keating in \cite{AK1, AK2, AK3, AK13} and extend their results.
The second aim of the paper is to derive asymptotic formulas for the mean values of quadratic Dirichlet $L$--functions
over the rational function field at the special point $s=1$ and as an immediate corollary to obtain the mean values
of the associated class numbers over function fields.

One of the main novelties of this paper is that we compute the second moment of quadratic Dirichlet $L$-functions associated
to monic irreducible polynomials of \textit{even} degree (the odd degree case was computed by Andrade-Keating in \cite{AK13}), and in this way we are able to go beyond of what is known in the number field case. The second innovation of this paper comes from the fact that, if the cardinality of the ground field is even then the task of average $L$-functions in function fields is much harder and, in this paper, we are able to handle this strenuous case and establish several mean values results of $L$-functions over function fields. (See next section for more details.)

\subsection{Zeta function of curves}
Let $\mathbb{F}_{q}$ be a finite field of odd cardinality, $\A=\mathbb{F}_{q}[T]$ the polynomial ring over $\mathbb{F}_{q}$
and we denote by $k=\mathbb{F}_{q}(T)$ the rational function field over $\mathbb{F}_{q}$.
We consider $C$ to be any smooth, projective, geometrically connected curve of genus $g\geq1$ defined over the finite field $\mathbb{F}_{q}$.
In this setting Artin \cite{Art} has defined the zeta function of the curve $C$ as
\begin{equation}
Z_{C}(u):=\exp\left(\sum_{n=1}^{\infty}N_{n}(C)\frac{u^{n}}{n}\right), \quad |u|<\frac{1}{q}
\end{equation}
where $N_{n}(C):=\mathrm{Card}(C(\mathbb{F}_{q}))$ is the number of points on $C$ with coordinates
in a field extension $\mathbb{F}_{q^{n}}$ of $\mathbb{F}_{q}$ of degree $n\geq1$ and $u=q^{-s}$.
It is well known that the zeta function attached to $C$ is a rational function
as proved by Weil \cite{Weil} and in this case is presented in the following form
\begin{equation}
Z_{C}(u)=\frac{L_{C}(u)}{(1-u)(1-qu)},
\end{equation}
where $L_{C}(u)\in\mathbb{Z}[u]$ is a polynomial of degree $2g$, called the \textbf{$L$--polynomial} of the curve $C$.
The Riemann-Roch theorem for function fields (see \cite[Theorem 5.4 and Theorem 5.9]{Ro02}) show us
that $L_{C}(u)$ satisfies the following functional equation:
\begin{equation}
L_{C}(u)=(qu^{2})^{g}L_{C}\left(\frac{1}{qu}\right).
\end{equation}
And the Riemann Hypothesis for curves over finite fields is a theorem in this setting,
which was established by Weil \cite{Weil} in 1948,
and it says that the zeros of $L_{C}(u)$ all lie on the circle $|u|=q^{-\frac{1}2}$, i.e.,
\begin{equation}
L_{C}(u)=\prod_{j=1}^{2g}(1-\alpha_{j}u), \quad \text{ with $|\alpha_{j}|=\sqrt{q}$\, for all $j$.}
\end{equation}

\subsection{Some Background on $\A = \mathbb{F}_{q}[T]$}
Let $\A^+$ denote the set of monic polynomials in $\A$
and $\mb P$ denote the set of monic irreducible polynomials in $\A$.
For a positive integer $n$ we denote by $\A_{n}^{+}$ to be the set of monic polynomials in $\A$ of degree $n$
and by $\mb P_{n}$ to be the set of monic irreducible polynomials in $\A$ of degree $n$.
Throughout this paper, a monic irreducible polynomial $P\in\mb{P}$ will be also called a ``prime" polynomial.
The \textit{norm} of a polynomial $f\in\mathbb{A}$ is defined to be $|f|:=q^{\mathrm{deg}(f)}$ for $f\neq0$, and $|f|=0$ for $f=0$.
The sign $sgn(f)$ of $f$ is the leading coefficient of $f$.
The zeta function of $\A$ is defined for $\mathrm{Re}(s)>1$ to be the following infinite series
\begin{equation}
\zeta_{\mathbb{A}}(s):=\sum_{f\in \A^{+}} |f|^{-s}
= \prod_{P\in\mb P}\left(1-|P|^{-s}\right)^{-1}, \quad \mathrm{Re}(s)>1.
\end{equation}
It is easy to show (see \cite[Chapter 2]{Ro02}) that the zeta function $\zeta_{\A}(s)$ is a very simple function and can be rewritten as
\begin{equation}
\zeta_{\A}(s)=\frac{1}{1-q^{1-s}}.
\end{equation}
The monic irreducible polynomials in $\A=\mathbb{F}_{q}[T]$ also satisfies the analogue of the Prime Number Theorem.
In other words we have the following Theorem (\cite[Theorem 2.2]{Ro02}).

\begin{thm}[Prime Polynomial Theorem] \label{thm:pnt}
Let $\pi_{\A}(n)$ denote the number of monic irreducible polynomials in $\A$ of degree $n$.
Then, we have
\begin{equation}
\pi_{\A}(n)=\#\mb{P}_{n}=\frac{q^{n}}{n}+O\left(\frac{q^{\tfrac{n}{2}}}{n}\right).
\end{equation}
\end{thm}

Hoffstein and Rosen \cite{HR} were one of the first to study mean values of $L$-functions over function fields.
In their beautiful paper, they established several mean values of $L$-series over the rational function field they considered averages over all monic polynomials, as well as the sum over square-free polynomials in $\mathbb{F}_{q}[T]$. But in their paper, they never consider mean values of $L$-functions associated to monic irreducible polynomials over $\mathbb{F}_{q}[T]$. In this paper, we investigate the problem of averaging $L$-functions over prime polynomials and we compute the first and the second moment of several families of $L$-functions, thus extending the pioneering work of Hoffstein and Rosen. It is also important to note that our methods are totally different from those used by Hoffstein and Rosen and are based on the use of the approximate functional
equation for function fields.

\section{Statement of Results}
\subsection{Odd characteristic case}
In this subsection we assume that $q$ is odd.
First we present some preliminary facts on quadratic Dirichlet $L$--functions
for the rational function field $k = \mathbb{F}_{q}(T)$
and for this we use Rosen's book \cite{Ro02} as a guide to the notations and definitions.
We also present the results of Andrade--Keating \cite{AK13}, which is the main inspiration for this article.

\subsubsection{\bf Quadratic Dirichlet $L$--function attached to $\chi_{D}$}
Fix a generator $\gamma$ of $\bF_{q}^{\times}$.
Let $\mb H$ be the set of non-constant square free polynomials $D$ in $\A$ with $sgn(D) \in \{1, \gamma\}$.
Then any quadratic extension $K$ of $k$ can be written uniquely as $K= K_{D} := k(\sqrt{D})$ for $D\in\mb H$.
The infinite prime $\infty = (1/T)$ of $k$ is ramified, splits, or is inert in $K_{D}$
according as $\deg(D)$ is odd, $\deg(D)$ is even and $sgn(D)=1$, or $\deg(D)$ is even and $sgn(D)=\gamma$.
Then $K_{D}$ is called {\it ramified imaginary}, {\it real}, or {\it inert imaginary}, respectively.
The genus $g_{D}$ of $K_{D}$ is given by
\begin{align}\label{genus-odd}
g_{D} = \left[\frac{\deg(D)+1}2\right] - 1 = \begin{cases}
\frac{1}{2}(\deg(D)-1) & \text{if $\deg(D)$ is odd,} \\ \frac{1}{2} \deg(D)-1 & \text{if $\deg(D)$ is even.}
\end{cases}
\end{align}

For $D\in\mb H$, let $\chi_{D}$ be the quadratic Dirichlet character modulo $D$
defined by the Kronecker symbol $\chi_{D}(f) = (\frac{D}{f})$ with $f\in\A$.
For more details about Dirichlet characters for polynomials over finite fields see \cite[Chapters 3, 4]{Ro02}.
The $L$--function associated to the character $\chi_{D}$ is defined by the following Dirichlet series
\begin{align}
L(s,\chi_{D}) := \sum_{f\in\A^{+}} \chi_{D}(f) |f|^{-s} = \prod_{P\in\mb P} \left(1-\chi_{D}(P) |P|^{-s}\right)^{-1},
\quad {\rm Re}(s) > 1.
\end{align}
From \cite[Propositions 4.3, 14.6 and 17.7]{Ro02},
we have that $L(s,\chi_{D})$ is a polynomial in $z=q^{-s}$ of degree $\deg(D)-1$.
Also we have that the quadratic Dirichlet $L$--function associated to $\chi_{D}$,
$\mathcal{L}(z,\chi_{D})=L(s,\chi_{D})$, has a ``trivial" zero at $z=1$ (resp. $z=-1$)
if and only if $\deg(D)$ is even and $sgn(D)=1$ (resp. $\deg(D)$ is even and $sgn(D)=\gamma$)
and so we can define the ``completed" $L$--function as
\begin{equation}\label{completed-L}
\mc{L}^{*}(z,\chi_{D}) = \begin{cases}
\mc{L}(z,\chi_{D}) & \text{ if $\deg(D)$ is odd,} \\
(1-z)^{-1} \mc{L}(z,\chi_{D}) & \text{ if $\deg(D)$ is even and $sgn(D) = 1$,} \\
(1+z)^{-1} \mc{L}(z,\chi_{D}) & \text{ if $\deg(D)$ is even and $sgn(D) = \gamma$,}
\end{cases}
\end{equation}
which is a polynomial of even degree $2 g_{D}$ and satisfies the functional equation
\begin{equation}\label{functional equation}
\mathcal{L}^{*}(z,\chi_{D})=(qz^{2})^{g_{D}}\mathcal{L}^{*}\left(\frac{1}{qz},\chi_{D}\right).
\end{equation}
In all cases, we have that
\begin{equation}
\mathcal{L}^{*}(z,\chi_{D})=L_{C_{D}}(z),
\end{equation}
where $L_{C_{D}}(z)$ is the numerator of the zeta function associated to the hyperelliptic curve given in the affine form by
\begin{equation}
C_{D} : y^{2}=D(T).
\end{equation}

The following proposition is quoted from Rudnick \cite{Ru}, and it is proved by using the explicit formula for $L(s,\chi_{P})$ and the Riemann Hypothesis for curves.

\begin{prop}\label{bound}
For any non-constant monic polynomial $f\in\A^{+}$, which is not a perfect square, we have
\begin{equation}
\Bigg|\sum_{P\in\mb P_{n}}\left(\frac{f}{P}\right)\Bigg|\ll\frac{\mathrm{deg}(f)}{n}q^{n/2}.
\end{equation}
\end{prop}

\subsubsection{\bf The Prime Hyperelliptic Ensemble}
We consider $\mb P_{n}$ as a probability space (ensemble) with the uniform probability measure attached to it.
So the expected value of any function $F$ on $\mb P_{n}$ is defined as
\begin{equation}
\left\langle F\right\rangle_{n}:=\frac{1}{\#\mb{P}_{n}}\sum_{P\in\mb{P}_{n}}F(P).
\end{equation}
Using Theorem \ref{thm:pnt} we have that
\begin{equation}
\frac{1}{\#\mb{P}_{n}}\sim\frac{\log_{q}|P|}{|P|}, \quad \text{as $n\rightarrow\infty$,}
\end{equation}
and thus we may write the expected value as
\begin{equation}\label{expected value}
\left\langle F(P)\right\rangle_{n}\sim\frac{\log_{q}|P|}{|P|}\sum_{P\in\mb{P}_{n}}F(P), \quad \text{as $n\rightarrow\infty$}.
\end{equation}

\subsubsection{\bf Main Results}
In this section we present the main theorems of this paper and we also state the previous results of Andrade--Keating \cite{AK13},
which is the main motivation for this paper.

From Andrade and Keating we have the following mean value theorem:

\begin{thm}[Andrade--Keating \cite{AK13}]
Let $\mathbb{F}_{q}$ be a fixed finite field of odd cardinality with $q\equiv 1 \pmod 4$.
Then for every $\varepsilon>0$ we have,
\begin{equation}
\sum_{P\in\mb{P}_{2g+1}}(\log_{q}|P|)L(\tfrac{1}{2},\chi_{P})
=\frac{|P|}{2}(\log_{q}|P|+1)+O\left(|P|^{\tfrac{3}{4}+\varepsilon}\right)
\end{equation}
and
\begin{equation}
\sum_{P\in\mb{P}_{2g+1}}L(\tfrac{1}{2},\chi_{P})^{2}=\frac{1}{24}\frac{1}{\zeta_{\A}(2)}|P|(\log_{q}|P|)^{2}+O\left(|P|(\log_{q}|P|)\right).
\end{equation}
\end{thm}

The results of Andrade and Keating correspond to the average of quadratic Dirichlet $L$--function
associated to the imaginary quadratic function field $k(\sqrt{P})$,
i.e., it is the function field analogue of the Problem \ref{mainprob} with $v=3$.
In this paper we extend the results of Andrade and Keating by establishing the corresponding asymptotic formulas
for the case of quadratic Dirichlet $L$--functions associated to the real quadratic function field $k(\sqrt{P})$,
which is the function field analogue of the Problem \ref{mainprob} with $v=1$. It is again worth noting that
in the classical case (number fields) only asymptotics formulas for the first moment of this family are known.
But in function fields, we can do better by establishing the second moment.
We also establish the corresponding asymptotic formulas for the case of quadratic Dirichlet $L$--functions
associated to the inert imaginary quadratic function field $k(\sqrt{\gamma P})$.
In addition, we derive asymptotic formulas for the mean values of quadratic Dirichlet prime $L$--functions at $s=1$ and $s=2$,
which are connected to the mean values of the ideal class numbers and to the cardinalities of second $K$-groups.
Our main results are presented below. 

A \textit{prime $L$--function} is the $L$--function associated to the quadratic character $\chi_{P}$ where $P$ is a prime polynomial.
For the first moment of prime $L$--functions, we have the following theorem.

\begin{thm}\label{F-M-Theorem}
Let $\mathbb{F}_{q}$ be a fixed finite field with $q$ being a power of an odd prime.
\begin{enumerate}
\item
For $s \in \CC$ with ${\rm Re}(s)\ge \frac12$, we have
\begin{align}
\sum_{P\in\mb P_{2g+1}} L(s,\chi_{P}) = I_{g}(s) \frac{|P|}{\log_{q}|P|}
+ \begin{cases}
O(|P|^{1-\frac{s}2}) & \text{ if ${\rm Re}(s) < 1$}, \vspace{0.1em}\\
O(|P|^{\frac{1}2} (\log_{q}|P|)) & \text{ if ${\rm Re}(s) \ge 1$,}
\end{cases}
\end{align}
as $g\to\infty$ with
\begin{align}
I_{g}(s) := \begin{cases}
g + 1 & \text{ if $s = \frac{1}2$,} \\
\zeta_{\A}(2s) (1- q^{(1+g)(1-2s)}) & \text{ if $\frac{1}2 \le {\rm Re}(s) < 1 (s \ne \frac{1}2)$.} \\
\zeta_{\A}(2s) & \text{ if ${\rm Re}(s) \ge 1$.}
\end{cases}
\end{align}

\item
For any $\epsilon>0$ and for $s \in \CC$ with ${\rm Re}(s)\ge \frac12$ and $|s-1|>\epsilon$, we have
\begin{align}
\sum_{P\in\mb P_{2g+2}} L(s,\chi_{P}) = J_{g}(s) \frac{|P|}{\log_{q}|P|}
+ \begin{cases}
O(|P|^{1-\frac{s}2}) & \text{ if ${\rm Re}(s) < 1$}, \vspace{0.1em} \\
O(|P|^{\frac{1}2} (\log_{q}|P|)) & \text{ if ${\rm Re}(s) \ge 1, \ \ \mathrm{and} \ \ ~(s\ne 1)$,}
\end{cases}
\end{align}
as $g\to\infty$ with
\begin{align}
J_{g}(s) := \begin{cases}
g + 1 + \zeta_{\A}(\frac{1}2) & \text{ if $s = \frac{1}2$,} \\
\zeta_{\A}(2s) J'_{g}(s) - \zeta_{\A}(2) J_{g}^{*}(s) & \text{ if $\frac{1}2 \le {\rm Re}(s) < 1 ~(s\ne \frac{1}2)$,}  \\
\zeta_{\A}(2s) J''_{g}(s) - \zeta_{\A}(2) J_{g}^{*}(s) & \text{ if $1 \le {\rm Re}(s) < \frac{3}2 ~(s\ne 1)$,} \\
\zeta_{\A}(2s) & \text{ if ${\rm Re}(s) \ge \frac{3}2$,}
\end{cases}
\end{align}
where
\begin{align}
J'_{g}(s) &= 1 - q^{([\frac{g}2+1])(1-2s)} + \tfrac{\zeta_{\A}(2-s)}{\zeta_{\A}(1+s)} \left(q^{(g-[\frac{g-1}2])(1-2s)} - q^{(g+1)(1-2s)}\right), \\
J''_{g}(s) &= 1 - q^{([\frac{g}2+1])(1-2s)} + \tfrac{\zeta_{\A}(2-s)}{\zeta_{\A}(1+s)} q^{(g-[\frac{g-1}2])(1-2s)}, \\
J_{g}^{*}(s) &= q^{[\frac{g}2]-(g+1)s} + \tfrac{\zeta_{\A}(2-s)}{\zeta_{\A}(1+s)} q^{[\frac{g-1}2]-gs},
\end{align}
and, for $s=1$, we have
\begin{align}
\sum_{P\in\mb P_{2g+2}} L(1,\chi_{P}) = \zeta_{\A}(2) \frac{|P|}{\log_{q}|P|} + O\left(|P|^{\frac{1}2} (\log_{q}|P|)\right).
\end{align}

\item
For $s \in \CC$ with ${\rm Re}(s)\ge \frac12$, we have
\begin{align}
\sum_{P\in\mb P_{2g+2}} L(s,\chi_{\gamma {P}}) = K_{g}(s) \frac{|P|}{\log_{q}|P|} + \begin{cases}
O(|P|^{1-\frac{s}2}) & \text{ if ${\rm Re}(s) < 1$}, \vspace{0.1em}\\
O(|P|^{\frac{1}2} (\log_{q}|P|)) & \text{ if ${\rm Re}(s) \ge 1$,}
\end{cases}
\end{align}
as $g\to\infty$ with
\begin{align}
K_{g}(s) := \begin{cases}
g+1+\zeta_{\A}(0) \zeta_{\A}(\frac{1}2)^{-1} & \text{ if $s=\frac{1}2$,} \\
\zeta_{\A}(2s) K'_{g}(s) + \zeta_{\A}(2) K^{*}_{g}(s) & \text{ if $\frac{1}2 \le {\rm Re}(s) < 1$ $(s\ne\frac{1}2)$,}  \\
\zeta_{\A}(2s) K''_{g}(s) + \zeta_{\A}(2) K^{*}_{g}(s) & \text{ if $1 \le {\rm Re}(s) < \frac{3}2$,}  \\
\zeta_{\A}(2s) & \text{ if $\frac{3}2 \le {\rm Re}(s)$,}
\end{cases}
\end{align}
where
\begin{align}
K'_{g}(s) &= 1 - q^{([\frac{g}2+1])(1-2s)} + \tfrac{1+q^{-s}}{1+q^{s-1}} \left(q^{(g-[\frac{g-1}2])(1-2s)} - q^{(g+1)(1-2s)}\right), \\
K''_{g}(s) &= 1 - q^{([\frac{g}2+1])(1-2s)} + \tfrac{1+q^{-s}}{1+q^{s-1}} q^{(g-[\frac{g-1}2])(1-2s)}, \\
K_{g}^{*}(s) &= (-1)^{g} q^{[\frac{g}2]-(g+1)s} + (-1)^{g+1}\tfrac{1+q^{-s}}{1+q^{s-1}} q^{[\frac{g-1}2]-gs}.
\end{align}
\end{enumerate}
\end{thm}

\begin{rem}
Note that in many of our estimates (e.g. when we use the $O$ and $\ll$ notations) the implied constant may depend on $q$.
\end{rem}

From Theorem \ref{F-M-Theorem} (2) and (3), we have the following corollary.
\begin{cor}
Let $q$ be a fixed power of an odd prime.
For every $\varepsilon>0$, we have
\begin{enumerate}
\item
\begin{align}
\sum_{P\in\mb P_{2g+2}} (\log_{q}|P|) L(\tfrac{1}2,\chi_{P})
= \frac{1}{2} \left(\log_{q}|P|  + 2\zeta_{\A}(\tfrac{1}2)\right) |P| + O\left(|P|^{\frac{3}4+\varepsilon}\right),
\end{align}
\item
\begin{align}
\sum_{P\in\mb P_{2g+2}} (\log_{q}|P|) L(\tfrac{1}2,\chi_{\gamma P})
= \frac{1}{2} \left(\log_{q}|P|  + 2\zeta_{\A}(0) \zeta_{\A}(\tfrac{1}2)^{-1}\right) |P| + O\left(|P|^{\frac{3}4+\varepsilon}\right),
\end{align}
\end{enumerate}
as $g\to\infty$.
\end{cor}

For the second moment of prime $L$--functions at $s=\frac{1}2$, we have the following theorem.

\begin{thm}\label{S-M-Theorem}
Let $q$ be a fixed power of an odd prime.
As $g\to\infty$, we have that
\begin{enumerate}
\item
\begin{align}
\sum_{P\in\mb P_{2g+2}} L(\tfrac{1}2,\chi_{P})^{2}
= \frac{1}{24 \zeta_{\A}(2)} |P| (\log_{q}|P|)^2  + O\left(|P| (\log_{q}|P|)\right),
\end{align}
\item
\begin{align}
\sum_{P\in\mb P_{2g+2}} L(\tfrac{1}2,\chi_{\gamma{P}})^{2}
= \frac{1}{24 \zeta_{\A}(2)} |P| (\log_{q}|P|)^2  + O\left(|P| (\log_{q}|P|)\right).
\end{align}
\end{enumerate}
\end{thm}

From Theorem \ref{F-M-Theorem} (2), (3), Theorem \ref{S-M-Theorem} and \eqref{expected value}, we obtain the following corollaries. 
Observe that corollary \ref{Non-Zero} below is about non-vanishing results of quadratic Dirichlet $L$-functions over function fields.

\begin{cor}\label{Odd-Ensemble}
With $q$ kept fixed power of an odd prime and $g\rightarrow\infty$, we have
\begin{align}
\left\langle L(\tfrac{1}{2},\chi_{P})\right\rangle_{2g+2}&\sim\frac{1}{2}\left(\log_{q}|P|+2\zeta_{\A}(\tfrac{1}2)\right), \\
\left\langle L(\tfrac{1}{2},\chi_{\gamma P})\right\rangle_{2g+2}
&\sim\frac{1}{2}\left(\log_{q}|P|+2\zeta_{\A}(0)\zeta_{\A}(\tfrac{1}2)^{-1}\right)
\end{align}
and
\begin{align}
\left\langle L(\tfrac{1}{2},\chi_{P})^{2}\right\rangle_{2g+2}&\sim\frac{1}{24\zeta_{\A}(2)}(\log_{q}|P|)^{3},\\
\left\langle L(\tfrac{1}{2},\chi_{\gamma P})^{2}\right\rangle_{2g+2}&\sim\frac{1}{24\zeta_{\A}(2)}(\log_{q}|P|)^{3}.
\end{align}
\end{cor}

\begin{cor}\label{Non-Zero}
With $q$ kept fixed power of an odd prime and $g\rightarrow\infty$, we have
\begin{equation}\label{Non-Zero-1}
\sum_{\substack{P\in\mb{P}_{2g+2} \\ L(\tfrac{1}{2},\chi_{P})\neq0}} 1 \gg \frac{|P|}{(\log_{q}|P|)^{2}}
\end{equation}
and
\begin{equation}\label{Non-Zero-2}
\sum_{\substack{P\in\mb{P}_{2g+2} \\ L(\tfrac{1}{2},\chi_{\gamma P})\neq0}} 1 \gg \frac{|P|}{(\log_{q}|P|)^{2}}.
\end{equation}
\end{cor}
\begin{proof}
We only give the proof of \eqref{Non-Zero-1}.
A similar argument will give the proof of \eqref{Non-Zero-2}.
From Theorem \ref{F-M-Theorem} (2) and Theorem \ref{S-M-Theorem} (1), we have
\begin{equation}\label{FM-001}
\sum_{P\in\mb{P}_{2g+2}}L(\tfrac{1}{2},\chi_{P})\sim c_{1}|P|
\end{equation}
and
\begin{equation}\label{SM-001}
\sum_{P\in\mb{P}_{2g+2}}L(\tfrac{1}{2},\chi_{P})^{2}\sim c_{2}|P|(\log_{q}|P|)^{2},
\end{equation}
where $c_{1}$ and $c_{2}$ are the constants given in the above theorems.
By applying Cauchy--Schwarz inequality we have that the number of monic irreducible polynomials $P\in\mb{P}_{2g+2}$
such that $L(\tfrac{1}{2},\chi_{P})\neq0$ exceeds the ratio of the square of the quantity in \eqref{FM-001}
to the quantity in \eqref{SM-001}.
\end{proof}

A simple computation shows that $K_{g}(1) = \zeta_{\A}(2)$.
For the mean value of prime $L$--functions at $s=1$, from Theorem \ref{F-M-Theorem}, we have the following:

\begin{thm}\label{fmoment-1}
Let $q$ be a fixed power of an odd prime.
For every $\varepsilon>0$, we have
\begin{enumerate}
\item
\begin{align}
\sum_{P\in\mb P_{2g+1}} L(1,\chi_{P}) = \zeta_{\A}(2) \frac{|P|}{\log_{q}|P|} + O\left(|P|^{\frac{1}2+\varepsilon}\right),
\end{align}
\item
\begin{align}\sum_{P\in\mb P_{2g+2}} L(1,\chi_{P}) = \zeta_{\A}(2) \frac{|P|}{\log_{q}|P|} + O\left(|P|^{\frac{1}2+\varepsilon}\right),
\end{align}
\item
\begin{align}
\sum_{P\in\mb P_{2g+2}} L(1,\chi_{\gamma {P}}) = \zeta_{\A}(2) \frac{|P|}{\log_{q}|P|}
+ O\left(|P|^{\frac{1}2+\varepsilon}\right),
\end{align}
\end{enumerate}
as $g\to\infty$. 
\end{thm}

For any non-constant square free polynomial $D\in\A$ with $sgn(D)\in \{1, \gamma\}$,
let $\OO_{D}$ be the integral closure of $\A$ in the quadratic function field $k(\sqrt{D})$.
Let $h_{D}$ be the ideal class number of $\OO_{D}$,
and $R_{D}$ be the regulator of $\OO_{D}$ if $\deg(D)$ is even and $sgn(D) = 1$.
We have a formula which connects $L(1, \chi_{D})$ with $h_{D}$ (\cite[Theorem 17.8A]{Ro02}):
\begin{align}\label{class-number-formula}
L(1, \chi_{D}) = \begin{cases}
\sqrt{q}~ |D|^{-\frac{1}2} h_{D} & \text{ if $\deg(D)$ is odd,} \vspace{0.1em}\\
(q-1) |D|^{-\frac{1}2} h_{D} R_{D} & \text{ if $\deg(D)$ is even and $sgn(D)=1$,} \vspace{0.1em}\\
\frac{1}{2} (q+1) |D|^{-\frac{1}2} h_{D} & \text{ if $\deg(D)$ is even and $sgn(D)=\gamma$.} \\
\end{cases}
\end{align}

By combining Theorem \ref{fmoment-1} and \eqref{class-number-formula}, we obtain the following corollary.

\begin{cor}
Let $q$ be a fixed power of an odd prime. 
For every $\varepsilon>0$, we have
\begin{enumerate}
\item
\begin{align}
\sum_{P\in\mb P_{2g+1}} h_{P} = \frac{\zeta_{\A}(2)}{\sqrt{q}} \frac{|P|^{\frac{3}2}}{\log_{q}|P|}
+ O\left(|P|^{1+\varepsilon}\right),
\end{align}
\item
\begin{align}
\sum_{P\in\mb P_{2g+2}} h_{P} R_{P} = q^{-1} \zeta_{\A}(2)^2 \frac{|P|^{\frac{3}2}}{\log_{q}|P|}
+ O\left(|P|^{1+\varepsilon}\right),
\end{align}
\item
\begin{align}
\sum_{P\in\mb P_{2g+2}} h_{P} = 2 q^{-1} \zeta_{\A}(3) \frac{|P|^{\frac{3}2}}{\log_{q}|P|} + O\left(|P|^{1+\varepsilon}\right),
\end{align}
\end{enumerate}
as $g\rightarrow\infty$.
\end{cor}

For any non-constant square free polynomial $D\in\A$ with $sgn(D)\in \{1, \gamma\}$,
let $K_{2}(\OO_{D})$ be the second $K$-group of $\OO_{D}$.
We have a formula which connects $L(2,\chi_{D})$ with the cardinality $\#K_{2}(\OO_{D})$ of $K_{2}(\OO_{D})$ (\cite[Proposition 2]{Ro95}):
\begin{align}\label{second-k-group}
L(2, \chi_{D}) = \begin{cases}
q^{\frac{3}2} |D|^{-\frac{3}2} \#K_{2}(\OO_{D}) & \text{ if $\deg(D)$ is odd,} \vspace{0.1em}\\
\frac{\zeta_{\A}(2)}{\zeta_{\A}(3)} q^{2} |D|^{-\frac{3}2} \#K_{2}(\OO_{D})
& \text{ if $\deg(D)$ is even and $sgn(D)=1$,} \vspace{0.1em}\\
\frac{\zeta_{\A}(3)^2}{\zeta_{\A}(2)\zeta_{\A}(5)} q^{2} |D|^{-\frac{3}2} \#K_{2}(\OO_{D})
& \text{ if $\deg(D)$ is even and $sgn(D)=\gamma$.} \\
\end{cases}
\end{align}

For the mean value of prime $L$--functions at $s=2$, as an application of Theorem \ref{F-M-Theorem}, we have the following.

\begin{thm}\label{fmoment-2}
Let $q$ be a fixed power of an odd prime.
For every $\varepsilon>0$, we have
\begin{enumerate}
\item
\begin{align}
\sum_{P\in\mb P_{2g+1}} L(2,\chi_{P}) = \zeta_{\A}(4) \frac{|P|}{\log_{q}|P|} + O\left(|P|^{\frac{1}2+\varepsilon}\right),
\end{align}
\item
\begin{align}
\sum_{P\in\mb P_{2g+2}} L(2,\chi_{P}) = \zeta_{\A}(4) \frac{|P|}{\log_{q}|P|} + O\left(|P|^{\frac{1}2+\varepsilon}\right),
\end{align}
\item
\begin{align}
\sum_{P\in\mb P_{2g+2}} L(2,\chi_{\gamma {P}}) = \zeta_{\A}(4) \frac{|P|}{\log_{q}|P|}
+ O\left(|P|^{\frac{1}2+\varepsilon}\right),
\end{align}
\end{enumerate}
as $g\to\infty$. 
\end{thm}

Putting together Theorem \ref{fmoment-2} and \eqref{second-k-group}, we obtain the following corollary.

\begin{cor}
Let $q$ be a fixed power of an odd prime.
For every $\varepsilon>0$, we have
\begin{enumerate}
\item
\begin{align}
\sum_{P\in\mb P_{2g+1}} \#K_{2}(\OO_{P}) = q^{-\frac{3}2} \zeta_{\A}(4) \frac{|P|^{\frac{5}2}}{\log_{q}|P|}
+ O\left(|P|^{2+\varepsilon}\right),
\end{align}
\item
\begin{align}
\sum_{P\in\mb P_{2g+2}} \#K_{2}(\OO_{P}) = q^{-2} \frac{\zeta_{\A}(3) \zeta_{\A}(4)}{\zeta_{\A}(2)}
\frac{|P|^{\frac{5}2}}{\log_{q}|P|} + O\left(|P|^{2+\varepsilon}\right),
\end{align}
\item
\begin{align}
\sum_{P\in\mb P_{2g+2}} \#K_{2}(\OO_{\gamma P}) = q^{-2} \frac{\zeta_{\A}(2) \zeta_{\A}(4) \zeta_{\A}(5)}{\zeta_{\A}(3)^2} \frac{|P|^{\frac{5}2}}{\log_{q}|P|}
+ O\left(|P|^{2+\varepsilon}\right),
\end{align}
\end{enumerate}
as $g\to\infty$. 
\end{cor}

\subsection{Even characteristic case}
We now will handle the more difficult case. In this section we assume that $q$ is a power of $2$.

\subsubsection{\bf Quadratic function fields of even characteristic}
The theory of quadratic function fields of even characteristic was first developed in \cite{CY08}
and we sketch below the basics on function fields of characteristic even.
Every separable quadratic extension $K$ of $k$ is of the form $K = K_{u} := k(x_{u})$,
where $x_{u}$ is a zero of $X^2+X+u=0$ for $u\in k$.
Let $\wp : k \to k$ be the additive homomorphism defined by $\wp(x) = x^2 + x$. 
Two extensions $K_{u}$ and $K_{v}$ are equal if and only if
$x_{v} = \alpha x_{u} + w, v = \alpha u + \wp(w)$ for some $\alpha\in\bF_{q}^{\times}$ and $w\in k$.
Hence, we can normalize $u$ to satisfy the following conditions (see \cite{HL10}):
\begin{align}\label{normalization-1}
&\qquad\qquad u = \sum_{i=1}^{m} \frac{Q_{i}}{P_{i}^{e_i}} + f(T), \\
& (P_{i}, Q_{i}) = 1 \text{ and } 2\nmid e_{i} \text{ for } 1 \le i \le m,  \\
&\qquad 2 \nmid \deg(f(T)) \text{ if } f(T) \in \bF_{q}[T]\setminus \bF_{q},
\end{align}
where $P_{i}\in\mb P$ are distinct and $Q_{i} \in \A$ with $\deg(Q_{i}) < \deg(P_{i}^{e_i})$ for $1 \le i \le m$.
In this case, the infinite prime $\infty = (1/T)$ is split, inert or ramified in $K_{u}$ according as
$f(T) = 0$, $f(T) \in \bF_{q}\setminus\wp(\bF_{q})$ or $f(T) \not\in\bF_{q}$.
Then the field $K_{u}$ is called \emph{real}, \emph{inert imaginary} or \emph{ramified imaginary}, respectively.
Let $\OO_{u}$ be the integral closure of $\A$ in $K_{u}$.
Then $\OO_{u} = \A + G_{u} x_{u}\A$, where $G_{u} := \prod_{i=1}^{m} P_{i}^{(e_{i}+1)/2}$,
and the discriminant of $\OO_{u}$ is $G_{u}^2 = \prod_{i=1}^{m} P_{i}^{e_{i}+1}$.
The local discriminant of $K_{u}$ at the infinite prime $\infty$ is $\infty^{\deg(f(T)) + 1}$
if $K_{u}$ is ramified imaginary and trivial otherwise.
Hence, the discriminant $D_{u}$ of $K_u$ is given by
\begin{align}
D_{u} = \begin{cases}
G_{u}^2 \cdot \infty^{\deg(f(T)) + 1} & \text{ if $K_{u}$ is ramified imaginary,} \\
G_{u}^2 & \text{ otherwise,}
\end{cases}
\end{align}
and, by the Hurwitz genus formula, the genus $g_{u}$ of $K_{u}$ is given by
\begin{align}
g_{u} = \frac{1}2 \deg(D_{u}) - 1.
\end{align}
In general, the normalization \eqref{normalization-1} is not unique.
Fix an element $\xi \in \bF_{q}\setminus \wp(\bF_{q})$.
Then every $u$ can be normalized uniquely to satisfy the following conditions:
\begin{align}\label{normalization-2}
u = \sum_{i=1}^{m} \sum_{j=0}^{\ell_{i}} \frac{Q_{ij}}{P_{i}^{2j+1}} + \sum_{i=1}^{n} \alpha_{i} T^{2i-1} + \alpha,
\end{align}
where $P_{i}\in\mb P$ are distinct, $Q_{ij}\in\A$ with $\deg(Q_{ij}) < \deg(P_{i})$, $Q_{i \ell_{i}} \ne 0$,
$\alpha \in \{0, \xi\}$, and $\alpha_{n} \ne 0$ for $n>0$.
Let $\widetilde{\mc F}$ be the set of such $u$'s above with $n=0$ and $\alpha=0$,
and $\widetilde{\mc F}'$ be the set of such $u$'s above with $n=0$ and $\alpha=\xi$.
Then, we see that $u\mapsto K_{u}$ defines an one-to-one correspondence between $\widetilde{\mc F}$ (resp. $\widetilde{\mc F}'$)
and the set of real (resp. inert imaginary) separable quadratic extensions of $k$.
Similarly, if we denote by $\widetilde{\mc H}$ the set of such $u$'s above with $n\ne 0$, then $u\mapsto K_{u}$ defines an one-to-one
correspondence between $\widetilde{\mc H}$ and the set of ramified imaginary separable quadratic extensions of $k$.

\subsubsection{\bf Hasse symbol and $L$-functions}
Let $P\in\mb P$.
For $u\in k$ which is $P$-integral, the Hasse symbol $[u, P)$ with values in $\bF_{2}$ is defined by
\begin{align}
[u, P) := \begin{cases}
0 & \text{ if $X^2+X\equiv u \bmod P$ is solvable in $\A$,} \\ 1 & \text{ otherwise.}
\end{cases}
\end{align}
For $N\in\A$ prime to the denominator of $u$, write $N = sgn(N) \prod_{i=1}^{s} P_{i}^{e_i}$,
where $P_{i} \in \mb P$ are distinct and $e_{i} \ge 1$, and define $[u, N)$ to be $\sum_{i=1}^{s} e_{i} [u, P_{i})$.

For $u\in k$ and $0\ne N \in\A$, we also define the quadratic symbol:
\begin{align}
\left\{\frac{u}{N}\right\} := \begin{cases}
(-1)^{[u, N)} & \text{ if $N$ is prime to the denominator of $u$,} \\ 0 & \text{ otherwise.}
\end{cases}
\end{align}
This symbol is clearly additive in its first variable, and multiplicative in the second variable.

For the quadratic extension $K_{u}$ of $k$, we associate a character $\chi_{u}$ on $\A^{+}$
which is defined by $\chi_{u}(f) = \{\frac{u}{f}\}$.
Let $L(s,\chi_{u})$ be the $L$-function  associated to the character $\chi_{u}$:
for $s\in\mb C$ with ${\rm Re}(s)\ge 1$,
\begin{align}
L(s,\chi_{u}) = \sum_{f\in\A^{+}} \chi_{u}(f) |f|^{-s}
= \prod_{P\in\mb P} \left(1-\chi_{u}(P)|P|^{-s}\right)^{-1}.
\end{align}
It is well known that $L(s,\chi_{u})$ is a polynomial in $q^{-s}$.
Letting $z=q^{-s}$, write $\mc L(z,\chi_{u}) = L(s, \chi_{u})$.
Then, $\mc L(z,\chi_{u})$ is a polynomial in $z$ of degree $2 g_{u} + \frac{1+(-1)^{\varepsilon(u)}}{2}$,
where $\varepsilon(u) = 1$ if $K_{u}$ is ramified imaginary and $\varepsilon(u) = 0$ otherwise.
Also we have that $\mathcal{L}(z,\chi_{u})$ has a ``trivial" zero at $z=1$ (resp. $z=-1$)
if and only if $K_{u}$ is real (resp. inert imaginary), so we can define the ``completed" $L$--function as
\begin{equation}\label{E-completed-L}
\mc{L}^{*}(z,\chi_{u}) = \begin{cases}
\mc{L}(z,\chi_{u}) & \text{ if $K_{u}$ is ramified imaginary,} \\
(1-z)^{-1} \mc{L}(z,\chi_{u}) & \text{ if $K_{u}$ is real,} \\
(1+z)^{-1} \mc{L}(z,\chi_{u}) & \text{ if $K_{u}$ is inert imaginary,}
\end{cases}
\end{equation}
which is a polynomial of even degree $2 g_{u}$ satisfying the functional equation
\begin{equation}\label{E-functional equation}
\mathcal{L}^{*}(z,\chi_{u})=(q z^{2})^{g_{u}}\mathcal{L}^{*}\left(\frac{1}{qz},\chi_{u}\right).
\end{equation}

\subsubsection{\bf Main results}
We are interested in the family of real, inert imaginary or ramified imaginary quadratic extensions $K_{u}$
of $k$ whose finite discriminant is a square of prime polynomial, i.e., $G_{u}\in\mb P$.
For any two subsets $U, V$ of $k$ and $w\in k$, we write
\begin{align}
U+V := \{u+v : u\in U \text{ and } v\in V\} \quad\text{ and }\quad U+w := \{u+w : u \in U\}.
\end{align}
Let $\mc F$ be the set of rational functions $u\in\widetilde{\mc F}$ whose denominator is a monic irreducible polynomial,
i.e., $u = \frac{A}{P}\in\widetilde{\mc F}$ with $P\in\mb P$ and $0 \ne A\in\A, \deg(A)<\deg(P)$, and $\mc F' = \mc F + \xi$.
Then, under the above correspondence $u\mapsto K_{u}$, $\mc F$ (resp. $\mc F'$) corresponds to
the set of real (resp. inert imaginary) separable quadratic extensions of $k$ whose discriminant is a square of prime polynomial.
For each positive integer $n$, let $\mc F_{n}$ be the set of rational functions $u = \frac{A}{P} \in \mc F$
such that $P \in \mb P_{n}$ and $\mc F'_{n} = \mc F_{n} + \xi$.
Then, under the correspondence $u\mapsto K_{u}$, $\mc F_{g+1}$ (resp. $\mc F'_{g+1}$) corresponds to the set of real (resp. inert imaginary)
separable quadratic extensions of $k$ whose discriminant is a square of prime polynomial and genus is $g$.

For each positive integer $s$, let $\mc G_{s}$ be the set of polynomials $F(T) \in \A_{2s-1}$ of the form
\begin{align}
F(T) = \alpha+\sum_{i=1}^{s}\alpha_{i}T^{2i-1}, \quad \alpha\in\{0, \xi\}, ~~~\alpha_{s} \ne 0.
\end{align}
Let $\mc G$ be the union of $\mc G_{s}$'s for $s \ge 1$ and $\mc H = \mc F + \mc G$.
Then, under the correspondence $u\mapsto K_{u}$, $\mc H$ corresponds to the set of ramified imaginary separable quadratic extensions of $k$
whose finite discriminant is a square of prime polynomial.
For integers $r,s\ge 1$, let $\mc H_{(r,s)} = \mc F_{r}+\mc G_{s}$.
Then, for each $u\in \mc H_{(r,s)}$, the corresponding field $K_{u}$ is a ramified imaginary imaginary of genus $r+s-1$.
For integer $n\ge 1$, let $\mc H_{n}$ be the union of $\mc H_{(r,n-r)}$'s for $1 \le r \le n-1$.
Then, under the correspondence $u\mapsto K_{u}$, $\mc H_{g+1}$ corresponds to the set of ramified imaginary separable quadratic extensions of $k$
whose finite discriminant is a square of prime polynomial and genus is $g$.

In this paper, we are interested in asymptotics for the sums (as $q$ is fixed and $g\to\infty$):
\begin{align}
\sum_{u\in\mc H_{g+1}} L(s,\chi_{u}), \hspace{0.3em}  \sum_{u\in\mc F_{g+1}} L(s,\chi_{u}), \hspace{0.3em}
\sum_{u\in\mc F'_{g+1}} L(s,\chi_{u}), \hspace{0.3em}\mathrm{Re}(s)\ge\frac{1}2
\end{align}
and
\begin{align}
\sum_{u\in\mc H_{g+1}} L(\tfrac{1}2,\chi_{u})^2, \hspace{0.3em} \sum_{u\in\mc F_{g+1}} L(\tfrac{1}2,\chi_{u})^2,
\hspace{0.3em} \sum_{u\in\mc F'_{g+1}} L(\tfrac{1}2,\chi_{u})^2.
\end{align}
For each $P\in\mb P$, let $\mc F_{P}$ be the set of rational functions $u \in \mc F$ whose denominator is $P$,
and $\mc F'_{P} := \mc F_{P} + \xi$.
Then $\mc F_{g+1}$ is disjoint union of the $\mc F_{P}$'s and $\mc F'_{g+1}$ is disjoint union of the $\mc F'_{P}$'s,
where $P$ runs over prime polynomials in $\mb P_{g+1}$.
Hence, we can write
\begin{align}
\sum_{u\in\mc F_{g+1}} L(s,\chi_{u}) = \sum_{P\in\mb P_{g+1}} \sum_{u\in\mc F_{P}} L(s,\chi_{u}), \quad
\sum_{u\in\mc F'_{g+1}} L(s,\chi_{u}) = \sum_{P\in\mb P_{g+1}} \sum_{u\in\mc F'_{P}} L(s,\chi_{u}),
\end{align}
and
\begin{align}
\sum_{u\in\mc F_{g+1}} L(\tfrac{1}2,\chi_{u})^2 = \sum_{P\in\mb P_{g+1}} \sum_{u\in\mc F_{P}}  L(\tfrac{1}2,\chi_{u})^2, \quad
\sum_{u\in\mc F'_{g+1}} L(\tfrac{1}2,\chi_{u})^{2} = \sum_{P\in\mb P_{g+1}} \sum_{u\in\mc F'_{P}} L(\tfrac{1}2,\chi_{u})^2.
\end{align}

For the first moment of such $L$--functions, we have the following theorem.

\begin{thm}\label{E-F-M-Theorem}
Let $\mathbb{F}_{q}$ be a fixed finite field with $q$ being a power of $2$.
\begin{enumerate}
\item
For $s \in \CC$ with ${\rm Re}(s)\ge \frac12$, we have
\begin{align}
\sum_{u\in\mc H_{g+1}} L(s,\chi_{u}) = 2\tilde{I}_{g}(s) \frac{q^{2g+1}}{g}
+ \begin{cases}
O(g^{-1} q^{2g}) & \text{ if $s=\frac{1}2$,} \vspace{0.2em}\\
O(g^{-2} q^{2g}) & \text{ if $s\ne\frac{1}2$,}
\end{cases}
\end{align}
as $g\rightarrow\infty$ with
\begin{align}
\tilde{I}_{g}(s) = \begin{cases}
g+1 & \text{ if $s=\frac{1}2$,} \vspace{0.2em}\\
\zeta_{\A}(2s) & \text{ if $s\ne\frac{1}2$.}
\end{cases}
\end{align}

\item
For any $\varepsilon>0$ and for $s \in \CC$ with ${\rm Re}(s)\ge \frac12$ and $|s-1|>\varepsilon$, we have
\begin{align}
\sum_{P\in\mb P_{g+1}}\sum_{u\in\mc F_{P}} L(s,\chi_{u}) = \tilde{J}_{g}(s) \frac{|P|^{2}}{\log_{q}|P|} + O\left(|P|^{\frac{3}{2}}\right)
\end{align}
as $g\rightarrow\infty$ with
\begin{align}
\tilde{J}_{g}(s) := \begin{cases}
g+1+\zeta_{\A}(\tfrac{1}2) & \text{ if $s=\frac{1}2$,} \\
\zeta_{\A}(2s) {J}'_{g}(s) - \zeta_{\A}(2) {J}^{*}_{g}(s) & \text{ if $\frac{1}2 \le {\rm Re}(s) < 1 ~(s\ne \frac1{2})$,}  \\
\zeta_{\A}(2s) & \text{ if ${\rm Re}(s) \ge 1$,}
\end{cases}
\end{align}
where ${J}'_{g}(s)$ and ${J}^{*}_{g}(s)$ are given in Theorem \ref{F-M-Theorem} (2) and, for $s=1$, we have
\begin{align}
\sum_{P\in\mb P_{g+1}} \sum_{u\in\mc F_{P}} L(1,\chi_{u}) = \zeta_{\A}(2) \frac{|P|^{2}}{\log_{q}|P|} + O\left(|P|^{\frac{3}{2}}\right).
\end{align}
\item
For $s \in \CC$ with ${\rm Re}(s)\ge \frac12$, we have
\begin{align}
\sum_{P\in\mb P_{g+1}}\sum_{u\in\mc F'_{P}} L(s,\chi_{u}) = \tilde{K}_{g}(s) \frac{|P|^{2}}{\log_{q}|P|} + O\left(|P|^{\frac{3}{2}}\right),
\end{align}
as $g\rightarrow\infty$ with
\begin{align}
\tilde{K}_{g}(s) := \begin{cases}
g+1+\zeta_{\A}(0)\zeta_{\A}(\frac1{2})^{-1} & \text{ if $s=\frac{1}2$, } \\
\zeta_{\A}(2s) {K}'_{g}(s) - \zeta_{\A}(2) {K}^{*}_{g}(s) & \text{ if $\frac{1}2 \le {\rm Re}(s) < 1 ~(s\ne \frac1{2})$,}  \\
\zeta_{\A}(2s) & \text{ if ${\rm Re}(s) \ge 1$,}
\end{cases}
\end{align}
where ${K}'_{g}(s)$ and ${K}^{*}_{g}(s)$ are given in Theorem \ref{F-M-Theorem} (3).
\end{enumerate}
\end{thm}

\begin{rem}
If $q$ is odd, the quadratic extension $k(\sqrt{\gamma D})$ of $k$ is also ramified imaginary
for any monic square-free polynomial $D$ of odd degree.
Under the changing of variable $T\mapsto \gamma T$, $k(\sqrt{\gamma D})$ becomes to $k(\sqrt{D})$.
Hence, we only consider the family $\{k(\sqrt{P}) : P\in\mb P_{2g+1}\}$ in Theorem \ref{F-M-Theorem} (1). However, if $q$ is even, we consider all separable ramified imaginary one.
This is the reason why the constant ``2" appears in Theorem \ref{E-F-M-Theorem} (1) and does not appear in Theorem \ref{F-M-Theorem} (1).
\end{rem}

For the second moment of $L$--functions at $s=\frac{1}2$, we have the following theorem.

\begin{thm}\label{E-S-M-Theorem}
Let $q$ be a fixed power of $2$.
As $g\to\infty$, we have
\begin{enumerate}
\item
\begin{align}
\sum_{u\in\mc H_{g+1}} L(\tfrac{1}2,\chi_{u})^{2} = \frac{2}{3 \zeta_{\A}(2)} g^2 q^{2g+1} + O\left((\log g) g q^{2g}\right),
\end{align}
\item
\begin{align}
\sum_{P\in\mb P_{g+1}} \sum_{u\in\mc F_{P}} L(\tfrac{1}2,\chi_{u})^2 = \frac{1}{3 \zeta_{\A}(2)} |P|^2 (\log_{q}|P|)^2 + O\left(|P|^2 (\log_{q}|P|)\right),
\end{align}
\item
\begin{align}
\sum_{P\in\mb P_{g+1}} \sum_{u\in\mc F'_{P}} L(\tfrac{1}2,\chi_{u})^{2} = \frac{1}{3 \zeta_{\A}(2)} |P|^2 (\log_{q}|P|)^2 + O\left(|P|^2 (\log_{q}|P|)\right).
\end{align}
\end{enumerate}
\end{thm}

\begin{rem}
There is a unique quadratic extension $k(\sqrt{P})$ of $k$ whose (finite) discriminant is $P$ in case of $q$ being odd,
but if $q$ is even, there are $\Phi(P) = |P|-1$ separable quadratic extensions $K_{u}$ of $k$ whose (finite) discriminant is $P^2$.
For this reason, $|P|^2$ appears in Theorem \ref{E-S-M-Theorem} whereas $|P|$ appears in Theorem \ref{S-M-Theorem}.
Considering this difference, we may regard Theorem \ref{E-S-M-Theorem} as an even characteristic analogue of Theorem \ref{S-M-Theorem}.
\end{rem}

As in odd characteristic case, we can consider $\mc H_{g+1}$, $\mc F_{g+1}$ and $\mc F'_{g+1}$ as probability spaces (ensembles)
with the uniform probability measure attached to them.
So the expected value of any function $F$ on $\mc H_{g+1}$, $\mc F_{g+1}$ or $\mc F'_{g+1}$ is defined as
\begin{align}
\left\langle F\right\rangle_{\mc H_{g+1}} &= \frac{1}{\#\mc H_{g+1}}\sum_{u\in\mc H_{g+1}}F(u), \\
\left\langle F\right\rangle_{\mc F_{g+1}} &= \frac{1}{\#\mc F_{g+1}}\sum_{u\in\mc F_{g+1}}F(u), \\
\left\langle F\right\rangle_{\mc F'_{g+1}} &= \frac{1}{\#\mc F'_{g+1}}\sum_{u\in\mc F'_{g+1}}F(u).
\end{align}
Since (see Lemma \ref{E-PNT})
\begin{align}
& \#\mc H_{g+1} = 2 \frac{q^{2g+1}}{g}+ O\left(\frac{q^{2g}}{g^2}\right), \\
& \#\mc F_{g+1} = \#\mc F'_{g+1} = \frac{q^{2g+2}}{g+1} + O\left(\frac{q^{\frac{3g}{2}}}{g}\right),
\end{align}
we have
\begin{align}
\frac{1}{\#\mc H_{g+1}} \sim \frac{g}{2 q^{2g+1}} ~~~\text{ and }~~~
\frac{1}{\#\mc F_{g+1}} = \frac{1}{\#\mc F'_{g+1}} \sim \frac{g+1}{q^{2g+2}} \quad \text{ as $g\to\infty$.}
\end{align}
From Theorems \ref{E-F-M-Theorem} and \ref{E-S-M-Theorem}, we get the following corollary.

\begin{cor}\label{Even-Ensemble}
With $q$ kept fixed power of $2$ and $g\rightarrow\infty$, we have
\begin{align}
\left\langle L(\tfrac{1}{2},\chi_{u})\right\rangle_{\mc H_{g+1}} &\sim g+1, \\
\left\langle L(\tfrac{1}{2},\chi_{u})\right\rangle_{\mc F_{g+1}} &\sim g+1+\zeta_{\A}(\tfrac{1}2),  \\
\left\langle L(\tfrac{1}{2},\chi_{u})\right\rangle_{\mc F'_{g+1}} &\sim g+1+\zeta_{\A}(0)\zeta_{\A}(\tfrac1{2})^{-1}
\end{align}
and
\begin{align}
\left\langle L(\tfrac{1}{2},\chi_{u})^2\right\rangle_{\mc H_{g+1}} &\sim \frac{g^3}{3 \zeta_{\A}(2)}, \\
\left\langle L(\tfrac{1}{2},\chi_{u})^2\right\rangle_{\mc F_{g+1}} &\sim  \frac{(g+1)^3}{3 \zeta_{\A}(2)},  \\
\left\langle L(\tfrac{1}{2},\chi_{u})^2\right\rangle_{\mc F'_{g+1}} &\sim \frac{(g+1)^3}{3 \zeta_{\A}(2)}.
\end{align}
\end{cor}

We follow the same reasoning as it is done in Corollary \ref{Non-Zero} with Theorems \ref{E-F-M-Theorem} and \ref{E-S-M-Theorem} to get the following corollary.

\begin{cor}\label{E-Non-Zero}
With $q$ kept fixed power of $2$ and $g\rightarrow\infty$, we have
\begin{align}
& \sum_{\substack{u\in\mc H_{g+1}\\ L(\frac{1}2, \chi_{u})\ne0}} 1 \gg \frac{q^{2g+1}}{g^2}, \\
& \sum_{P\in\mb P_{g+1}} \sum_{\substack{u\in\mc F_{P}\\ L(\frac{1}2, \chi_{u})\ne0}} 1 \gg \left(\frac{|P|}{\log_{q}|P|}\right)^{2}, \\
& \sum_{P\in\mb P_{g+1}} \sum_{\substack{u\in\mc F'_{P}\\ L(\frac{1}2, \chi_{u})\ne0}} 1 \gg \left(\frac{|P|}{\log_{q}|P|}\right)^{2}.
\end{align}
\end{cor}

From Theorem \ref{E-F-M-Theorem}, we have the following result concerning the first moment of $L$--functions at $s=1$.
\begin{thm}\label{E-fmoment-1}
Let $q$ be a fixed power of $2$.
As $g\to\infty$, we have
\begin{enumerate}
\item
\begin{align}
\sum_{u \in \mc H_{g+1}} L(1,\chi_{u}) = 2 \zeta_{\A}(2) \frac{q^{2g+1}}{g} + O\left(\frac{q^{2g}}{g^2}\right),
\end{align}
\item
\begin{align}
\sum_{P\in\mb P_{g+1}} \sum_{u\in\mc F_{P}} L(1,\chi_{u})  = \zeta_{\A}(2) \frac{|P|^{2}}{\log_{q}|P|} + O\left(|P|^{\frac{3}{2}}\right),
\end{align}
\item
\begin{align}
\sum_{P\in\mb P_{g+1}} \sum_{u\in\mc F'_{P}} L(1,\chi_{u}) = \zeta_{\A}(2) \frac{|P|^{2}}{\log_{q}|P|} + O\left(|P|^{\frac{3}{2}}\right).
\end{align}
\end{enumerate}
\end{thm}

For any $u\in \mc H_{g+1} \cup \mc F_{g+1} \cup \mc F'_{g+1}$, we have a formula
which connects $L(1, \chi_{u})$ and the class number $h_{u}$ of $\OO_{u}$ (\cite[Theorem 5.2]{CY08}):
\begin{align}\label{class-number-2}
L(1, \chi_{u}) = \begin{cases}
q^{-g} h_{u}  & \text{ if $u \in \mc H_{g+1}$,} \\
\zeta_{\A}(2)^{-1} q^{-g} h_{u} R_{u} & \text{ if $u \in \mc F_{g+1}$,} \\
\frac{1}{2}\zeta_{\A}(2) \zeta_{\A}(3)^{-1} q^{-g} h_{u} & \text{ if $u \in \mc F'_{g+1}$,}
\end{cases}
\end{align}
where $R_{u}$ is the regulator of $\OO_{u}$ if $u\in\mc F_{g+1}$.

By combining Theorem \ref{E-fmoment-1} and equation \eqref{class-number-2}, we have the following corollary.

\begin{cor}
Let $q$ be a fixed power of $2$.
As $g\to\infty$, we have
\begin{enumerate}
\item
\begin{align}
\sum_{u \in \mc H_{g+1}} h_{u} = 2 \zeta_{\A}(2) \frac{q^{3g+1}}{g} + O\left(\frac{q^{3g}}{g^2}\right),
\end{align}
\item
\begin{align}
\sum_{P\in\mb P_{g+1}}\sum_{u\in\mc F_{P}} h_{u} R_{u} = \zeta_{\A}(2)^2 q^{-1} \frac{|P|^3}{\log_{q}|P|} + O\left(|P|^{\frac{5}{2}}\right),
\end{align}
\item
\begin{align}
\sum_{P\in\mb P_{g+1}}\sum_{u\in\mc F'_{P}} h_{u} = 2 \zeta_{\A}(3) q^{-1} \frac{|P|^3}{\log_{q}|P|} + O\left(|P|^{\frac{5}{2}}\right).
\end{align}
\end{enumerate}
\end{cor}


\section{``Approximate" functional equations of $L$--functions}
For any separable quadratic extension $K$ of $k$, let $\chi_{K}$ denote the character $\chi_{D}$ if $q$ is odd and $K=k(\sqrt{D})$,
where $D$ is a non-constant square free polynomials $D\in\A$ with $sgn(D) \in \{1, \gamma\}$,
or the character $\chi_{u}$ if $q$ is even and $K = K_{u}$, where $u\in k$ is normalized as in \eqref{normalization-2}.
Let $L(s,\chi_{K})$ be the $L$--function associated to $\chi_{K}$.
Then $L(s,\chi_{K})$ is a polynomial in $z = q^{-s}$ of degree $\delta_{K} = 2 g_{K} + \frac{1}{2}(1+(-1)^{\varepsilon(K)})$,
where $g_{K}$ is the genus of $K$, $\varepsilon(K) = 1$ if $K$ is ramified imaginary and $\varepsilon(K) = 0$ otherwise.
Write
\begin{align}
L(s,\chi_{K}) = \sum_{n=0}^{\delta_{K}} A_{K}(n) q^{-ns} \quad\text{ with $A_{K}(n) := \sum_{f\in\A_{n}^{+}} \chi_{K}(f)$}.
\end{align}

\begin{lem}\label{A-FE}
\begin{enumerate}
\item
If $K$ is ramified imaginary, then we have
\begin{align}\label{A-FE-Imaginary}
L(s,\chi_{K}) = \sum_{n=0}^{g_{K}} A_{K}(n) q^{-sn} + q^{(1-2s)g_{K}} \sum_{n=0}^{g_{K}-1} A_{K}(n) q^{(s-1)n}.
\end{align}
\item
If $K$ is real, then we have
\begin{align}\label{A-FE-Real}
L(s,\chi_{K}) = \sum_{n=0}^{g_{K}} A_{K}(n) q^{-sn} - q^{-(g_{K}+1)s} \sum_{n=0}^{g_{K}} A_{K}(n) + H_{K}(s),
\end{align}
where $H_{K}(1) := \zeta_{\A}(2)^{-1} q^{-g_{K}}\dis\sum_{n=0}^{g_{K}-1} \left(g_{K}-n\right) A_{K}(n)$ and, for $s\ne 1$,
\begin{align}
\hspace{2em}H_{K}(s) := q^{(1-2s)g_{K}} \frac{\zeta_{\A}(2-s)}{\zeta_{\A}(1+s)} \sum_{n=0}^{g_{K}-1} q^{(s-1)n} A_{K}(n)
- q^{-s g_{K}} \frac{\zeta_{\A}(2-s)}{\zeta_{\A}(1+s)} \sum_{n=0}^{g_{K}-1} A_{K}(n).
\end{align}
\item
If $K$ is inert imaginary, then we have
\begin{align}\label{A-FE-Inert}
L(s,\chi_{K}) &= \sum_{n=0}^{g_{K}} A_{K}(n) q^{-sn} + q^{-(\delta_D+1)s} \sum_{n=0}^{g_{K}}(-1)^{n+g_{K}} A_{K}(n) \nonumber \\
&\hspace{1em} + \left(\frac{1+q^{-s}}{1+q^{s-1}}\right) q^{(1-2s)g_{K}} \sum_{n=0}^{g_{K}-1} A_{K}(n) q^{(s-1)n}  \nonumber \\
&\hspace{1em} + \left(\frac{1+q^{-s}}{1+q^{s-1}}\right) q^{-sg_{K}} \sum_{n=0}^{g_{K}-1} (-1)^{n+g_{K}+1} A_{K}(n).
\end{align}
\end{enumerate}
\end{lem}
\begin{proof}
Write
\begin{equation}
\mc L(z,\chi_{K}) = \sum_{n=0}^{\delta_{K}} A_{K}(n) z^{n} \quad\text{ and }\quad
\mc L^{*}(z,\chi_{K}) = \sum_{n=0}^{2g_{K}} A_{K}^{*}(n) z^{n}.
\end{equation}
By \eqref{completed-L} and \eqref{E-completed-L}, we have
\begin{align}\label{C-O}
A_{K}^{*}(n) = \begin{cases}
A_{K}(n) & \text{ if $K$ is ramified imaginary,} \\
\sum_{i=0}^{n} A_{K}(i) & \text{ if $K$ is real,} \\
\sum_{i=0}^{n} (-1)^{n-i} A_{K}(i) & \text{ if $K$ is inert imaginary.}
\end{cases}
\end{align}
By substituting $\mc L^{*}(z,\chi_{K}) = \sum_{n=0}^{2g_{K}} A_{K}^{*}(n) z^{n}$ into the functional equation \eqref{functional equation} (or \eqref{E-completed-L})
\begin{align}
\sum_{n=0}^{2g_{K}} A_{K}^{*}(n) z^{n} = \sum_{n=0}^{2g_{K}} A_{K}^{*}(n) q^{g_{K}-n} z^{2g_{K}-n}
= \sum_{n=0}^{2g_{K}} A_{K}^{*}(2g_{K}-n) q^{n-g_{K}} z^{n},
\end{align}
and equating coefficients, we have
\begin{align}
A_{K}^{*}(n) = A_{K}^{*}(2g_{K}-n) q^{n-g_{K}} \quad\text{ or }\quad A_{K}^{*}(2g_{K}-n) = A_{K}^{*}(n) q^{g_{K}-n}.
\end{align}
So we can write $\mc L^{*}(z,\chi_{K})$ as
\begin{align}\label{Approximate-FE}
\mc L^{*}(z,\chi_{K}) = \sum_{n=0}^{g_{K}} A_{K}^{*}(n) z^{n}
+ q^{g_{K}} z^{2g_{K}} \sum_{n=0}^{g_{K}-1} A_{K}^{*}(n) q^{-n} z^{-n}.
\end{align}
If $K$ is ramified imaginary, since $\mc L(z,\chi_{K}) = \mc L^{*}(z,\chi_{K})$,
we have that \eqref{A-FE-Imaginary} follows immediately from \eqref{Approximate-FE}.
Suppose that $K$ is real.
By \eqref{C-O} and \eqref{Approximate-FE}, we have
\begin{align}\label{A-FE-001}
\mc L^{*}(z,\chi_{K}) &= \sum_{n=0}^{g_{K}} \left(\sum_{i=0}^{n} A_{K}(i)\right) z^{n}
+ q^{g_{K}} z^{2g_{K}} \sum_{n=0}^{g_{K}-1} \left(\sum_{j=0}^{n} A_{K}(j)\right) q^{-n} z^{-n} \nonumber \\
&= \sum_{n=0}^{g_{K}} \left(\frac{z^{n}-z^{g_{K}+1}}{1-z}\right) A_{K}(n) + H^{*}(z),
\end{align}
where $H^{*}(q^{-1}) := q^{-g_{K}}\dis\sum_{n=0}^{g_{K}-1} \left(g_{K}-n\right) A_{K}(n)$ and, for $z\ne q^{-1}$,
\begin{align}
H^{*}(z) := \dfrac{q^{g_{K}} z^{2g_{K}}}{1-q^{-1} z^{-1}} \dis\sum_{n=0}^{g_{K}-1} q^{-n} z^{-n} A_{K}(n)
- \dfrac{z^{g_{K}}}{1-q^{-1} z^{-1}} \dis\sum_{n=0}^{g_{K}-1} A_{K}(n).
\end{align}
By multiplying $(1-z)$ on \eqref{A-FE-001} and putting $z=q^{-s}$, we get \eqref{A-FE-Real}.
Finally, consider the case that $K$ is inert imaginary.
By \eqref{C-O} and \eqref{Approximate-FE}, we have
\begin{align}\label{A-FE-002}
\mc L^{*}(z,\chi_{K}) &= \sum_{n=0}^{g_{K}} \left(\sum_{i=0}^{n} (-1)^{n-i} A_{K}(i)\right) z^{n}
+ q^{g_{K}} z^{2g_{K}} \sum_{n=0}^{g_{K}-1} \left(\sum_{j=0}^{n} (-1)^{n-j} A_{K}(j)\right) q^{-n} z^{-n} \nonumber \\
&\hspace{-1.5em}= \frac{1}{1+z} \sum_{n=0}^{g_{K}} A_{K}(n) z^{n} + \frac{z^{g_{K}+1}}{1+z} \sum_{n=0}^{g_{K}}(-1)^{n+g_{K}} A_{K}(n) \nonumber\\
&\hspace{-1em}+ \frac{q^{g_{K}} z^{2g_{K}}}{1+q^{-1} z^{-1}} \sum_{n=0}^{g_{K}-1} A_{K}(n) q^{-n} z^{-n}
+ \frac{z^{g_{K}}}{1+q^{-1} z^{-1}} \sum_{n=0}^{g_{K}-1} (-1)^{n+g_{K}+1} A_{K}(n).
\end{align}
By multiplying $(1+z)$ on \eqref{A-FE-002} and putting $z=q^{-s}$, we get \eqref{A-FE-Inert}.
\end{proof}

Write
\begin{align}
L(s,\chi_{K})^2 = \sum_{n=0}^{2\delta_{K}} B_{K}(n) q^{-ns}
\quad\text{ with $B_{K}(n) := \sum_{f\in\A_{n}^{+}} d(f)\chi_{K}(f)$},
\end{align}
where $d(f)$ denotes the divisor function on $\A^{+}$:
\begin{align}
d(f) := \sum_{\substack{N\in\A^{+}\\ N|f}} 1.
\end{align}

\begin{lem}\label{A-FES}
\begin{enumerate}
\item
If $K$ is ramified imaginary, then we have
\begin{align}\label{A-FES-Imaginary}
L(\tfrac{1}2,\chi_{K})^{2} = \sum_{n=0}^{2g_{K}} B_{K}(n) q^{-\frac{n}2} + \sum_{n=0}^{2g_{K}-1} B_{K}(n) q^{-\frac{n}2}.
\end{align}
\item
If $K$ is real, then we have
\begin{align}\label{A-FES-Real}
L(\tfrac{1}2,\chi_{D})^{2} &= \sum_{n=0}^{2g_{K}} B_{K}(n) q^{-\frac{n}{2}}
+ \sum_{n=0}^{2g_{K}-1} B_{K}(n) q^{-\frac{n}{2}}  \nonumber \\
&\quad- q^{-(g_{K}+\frac12)} \sum_{n=0}^{2g_{K}} B_{K}(n) - q^{-g_{K}} \sum_{n=0}^{2g_{K}-1} B_{K}(n) \nonumber \\
&\quad - \zeta_{\A}(\tfrac{3}2)^{-1} q^{-(g_{K}+\frac12)} \sum_{n=0}^{2g_{K}} (2g_{K}+1-n) B_{K}(n) \nonumber \\
&\quad - \zeta_{\A}(\tfrac{3}2)^{-1} q^{-g_{K}}\sum_{n=0}^{2g_{K}-1} (2g_{K}-n) B_{K}(n).
\end{align}
\item
If $K$ is inert imaginary, then we have
\begin{align}\label{A-FES-Inert}
L(\tfrac{1}2,\chi_{D})^{2} &= \sum_{n=0}^{2g_{K}} B_{K}(n) q^{-\frac{n}{2}}
+ \sum_{n=0}^{2g_{K}-1} q^{-\frac{n}{2}} B_{K}(n) \nonumber\\
&\quad + q^{-(g_{K}+\frac12)} \sum_{n=0}^{2g_{K}} (-1)^n B_{K}(n)
+ q^{-g_{K}} \sum_{n=0}^{2g_{K}-1} (-1)^{n} B_{K}(n)  \nonumber\\
&\quad + \frac{\zeta_{\A}(\tfrac{3}2)}{\zeta_{\A}(2)} q^{-(g_{K}+\frac12)} \sum_{n=0}^{2g_{K}} (-1)^n (2g_{K}+1-n) B_{K}(n)  \nonumber \\
&\quad + \frac{\zeta_{\A}(\tfrac{3}2)}{\zeta_{\A}(2)} q^{-g_{K}}\sum_{n=0}^{2g_{K}-1} (-1)^{n} (2g_{K}-n) B_{K}(n).
\end{align}
\end{enumerate}
\end{lem}
\begin{proof}
Write
\begin{align}
\mc L(z,\chi_{D})^2 = \sum_{n=0}^{2\delta_{K}} B_{K}(n) z^{n} \quad\text{ and }\quad
\mc L^{*}(z,\chi_{D})^2 = \sum_{n=0}^{4 g_{K}} B_{K}^{*}(n) z^{n}.
\end{align}
By \eqref{completed-L} and \eqref{E-completed-L}, we have
\begin{align}\label{C-O-2}
B_{K}^{*}(n) = \begin{cases}
B_{K}(n) & \text{ if $K$ is ramified imaginary,} \\
\sum_{i=0}^{n} (n+1-i) B_{K}(i) & \text{ if $K$ is real,} \\
\sum_{i=0}^{n} (-1)^{n+i} (n+1-i) B_{K}(i) & \text{ if $K$ is inert imaginary.}
\end{cases}
\end{align}
From the functional equation
\begin{align}
\mathcal{L}^{*}(z,\chi_{K})^{2} = (qz^{2})^{2g_{K}} \mathcal{L}^{*}\left(\frac{1}{qz},\chi_{K}\right)^2,
\end{align}
we get
\begin{align}
B_{K}^{*}(n) = B_{K}^{*}(4 g_{K}-n) q^{n-2g_{K}} \quad\text{ or }\quad B_{K}^{*}(4 g_{K}-n)
= q^{2g_{K}-n} B_{K}^{*}(n).
\end{align}
Then we have
\begin{align}\label{A-FE-Square}
\mathcal{L}^{*}(z,\chi_{K})^{2}
= \sum_{n=0}^{2g_{K}} B_{K}^{*}(n) z^{n}
+ q^{2g_{K}} z^{4g_{K}} \sum_{n=0}^{2g_{K}-1} B_{K}^{*}(n) q^{-n} z^{-n}.
\end{align}
If $K$ is ramified imaginary, since $\mc L(z,\chi_{K})^2 = \mc L^{*}(z,\chi_{K})^2$, we have that
\eqref{A-FES-Imaginary} follows immediately from \eqref{A-FE-Square}.
Suppose that $K$ is real.
By \eqref{C-O-2} and \eqref{A-FE-Square}, we have
\begin{align}\label{AFES-001}
\mathcal{L}^{*}(q^{-\frac{1}2},\chi_{K})^{2}
&= \sum_{n=0}^{2g_{K}} \left\{\frac{q^{-\frac{n}{2}}- q^{-(g_{K}+\frac12)}}{(1-q^{-\frac12})^2}
-\frac{(2g_{K}+1-n)q^{-(g_{K}+\frac12)}}{1-q^{-\frac12}}\right\} B_{K}(n) \nonumber \\
&\qquad + \sum_{n=0}^{2g_{K}-1}\left\{\frac{q^{-\frac{n}{2}}-q^{-g_{K}}}{(1-q^{-\frac12})^2}
-\frac{(2g_{K}-n)q^{-g_{K}}}{1-q^{-\frac12}}\right\} B_{K}(n).
\end{align}
Multiplying $(1-q^{-\frac12})^2$ on \eqref{AFES-001}, we get \eqref{A-FES-Real}.
Suppose that $K$ is inert imaginary.
By \eqref{C-O-2} and \eqref{A-FE-Square}, we have
\begin{align}\label{AFES-002}
\mathcal{L}^{*}(q^{-\frac{1}2},\chi_{K})^{2}
&= \sum_{n=0}^{2g_{K}} \left\{\frac{q^{-\frac{n}{2}}+ (-1)^{n} q^{-(g_{K}+\frac12)}}{(1+q^{-\frac12})^2}
+\frac{(-1)^{n}(2g_{K}+1-n)q^{-(g_{K}+\frac12)}}{1+q^{-\frac12}}\right\} B_{K}(n) \nonumber \\
&\qquad + \sum_{n=0}^{2g_{K}-1}\left\{\frac{q^{-\frac{n}{2}}+(-1)^{n}q^{-g_{K}}}{(1+q^{-\frac12})^2}
+\frac{(-1)^{n}(2g_{K}-n)q^{-g_{K}}}{1+q^{-\frac12}}\right\} B_{K}(n).
\end{align}
Multiplying $(1+q^{-\frac12})^2$ on \eqref{AFES-002}, we get \eqref{A-FES-Inert}.
\end{proof}

\section{First moment of prime $L$--functions}
\subsection{Odd characteristic case}
In this subsection, we give a proof of Theorem \ref{F-M-Theorem}.
In \S\ref{Odd-F-M-SS1}, we obtain several results of the contribution of squares and of non-squares,
which will be used to calculate the first moment of prime $L$-functions
in \S\ref{Odd-F-M-SS3}, \S\ref{Odd-F-M-SS4} and \S\ref{Odd-F-M-SS5}.

In \S\ref{Odd-F-M-SS1}, $\mb H_{g}$ will denote $\mb P_{2g+1}$ or $\mb P_{2g+2}$ for any positive integer $g$.

\subsubsection{\bf Preparations for the proof}\label{Odd-F-M-SS1}
We first consider the contribution of squares.
We will use the Prime polynomial Theorem (Theorem \ref{thm:pnt}) in the following form
\begin{align}\label{thm:pnt-2}
\sum_{P\in\mb H_{g}} 1 = \frac{|P|}{\log_{q}|P|} + O\left(\frac{|P|^{\frac{1}2}}{\log_{q}|P|}\right).
\end{align}

\begin{prop}\label{Contribution-Square-001}
\begin{enumerate}
\item
For $s\in\mb C$ with ${\rm Re}(s) \ge \frac{1}2$, we have
\begin{align}
\hspace{3em} \sum_{P\in\mb H_{g}} \sum_{n=0}^{g} (\pm 1)^{n} q^{-sn} \sum_{\substack{f\in\A_{n}^+\\ f=\square}} \chi_{P}(f)
= A_{g}(s) \frac{|P|}{\log_{q}|P|} + O\left(|P|^{\frac{1}2}\right),
\end{align}
where
\begin{align}
A_{g}(s) = \begin{cases}
[\frac{g}2]+1 & \text{ if $s=\frac12$,} \\
\zeta_{\A}(2s) (1-q^{([\frac{g}2]+1)(1-2s)}) & \text{ if $s\ne \frac12$. }
\end{cases}
\end{align}
\item
For $s\in\mb C$ with ${\rm Re}(s) \ge \frac{1}2$, we have
\begin{align}
\hspace{2em} q^{(1-2s)g} \sum_{P\in\mb H_{g}} \sum_{n=0}^{g-1} (\pm 1)^{n} q^{(s-1)n} \sum_{\substack{f\in\A_{n}^+\\ f=\square}} \chi_{P}(f)
= B_{g}(s) \frac{|P|}{\log_{q}|P|} + O\left(|P|^{\frac{3}4-\frac{s}2}\right),
\end{align}
where
\begin{align}
B_{g}(s) = \begin{cases}
[\frac{g-1}2] + 1 & \text{ if } s = \frac12,  \vspace{0.2em}\\
\zeta_{\A}(2s) \big\{q^{(g-[\frac{g-1}2])(1-2s)} - q^{(g+1)(1-2s)}\big\} & \text{ if } s \ne \frac12.
\end{cases}
\end{align}
\item
Let $h \in \{g-1, g\}$.
For $s\in\mb C$ with ${\rm Re}(s) \ge \frac{1}2$, we have
\begin{align}
\hspace{3em} q^{-(h+1)s} \sum_{P\in\mb H_{g}}\sum_{n=0}^{h} \sum_{\substack{f\in\A_{n}^+\\ f=\square}} \chi_{P}(f)
= C_{h}(s) \frac{|P|}{\log_{q}|P|} + O\left(|P|^{\frac{3}4-\frac{s}2}\right),
\end{align}
where $C_{h}(s) = \zeta_{\A}(2) q^{-(h+1)s-1} (q^{[\frac{h}2]+1}-1)$.
\item
For $s\in\mb C$ with ${\rm Re}(s) \ge \frac{1}2$, we have
\begin{align}
\hspace{2em} \zeta_{\A}(2)^{-1} q^{-g} \sum_{P\in\mb P_{2g+2}}\sum_{n=0}^{g-1} (g-n)
\sum_{\substack{f\in\A_{n}^+\\ f=\square}} \chi_{P}(f)
= B(g) \frac{|P|^{\frac{1}2}}{\log_{q}|P|} + O\left(|P|^{\frac{1}4}\right),
\end{align}
where $B(g) = q^{([\frac{g-1}2]+1)} \left\{2\zeta_{\A}(2)- \tfrac{1}{2}(1+(-1)^{g+1})\right\} - g - 2 \zeta_{\A}(2)$.
\end{enumerate}
\end{prop}
\begin{proof}
(1) For any $P\in\mb H_{g}$ and $L\in\A_{l}^{+}$ with $l \le g$, we have $\chi_{P}(L^2) = 1$.
By \eqref{thm:pnt-2}, we have
\begin{align}\label{Contribution-Square-001-1}
\sum_{P\in\mb H_{g}} \sum_{n=0}^{g} (\pm 1)^{n} q^{-sn} \sum_{\substack{f\in\A_{n}^+\\ f=\square}} \chi_{P}(f)
&= \sum_{l=0}^{[\frac{g}2]} q^{-2ls} \sum_{L\in\A_{l}^+} \sum_{P\in\mb H_{g}} \chi_{P}(L^2)
= \sum_{l=0}^{[\frac{g}2]} q^{(1-2s)l} \sum_{P\in\mb H_{g}} 1  \nonumber \\
&= A_{g}(s) \frac{|P|}{\log_{q}|P|} + O\left(A_{g}(s) \frac{|P|^{\frac{1}2}}{\log_{q}|P|}\right).
\end{align}
Since $A_{g}(s) \ll g$, the error term in \eqref{Contribution-Square-001-1} is $\ll |P|^{\frac{1}2}$.
Hence, we get the result.
The proofs of (2), (3) and (4) are similar as that of (1).
\end{proof}

Now, we consider the contribution of non-squares.
For any non-constant monic polynomial $f$, which is not perfect square, we can reformulate Proposition \ref{bound} as follow:
\begin{equation}\label{bound-2}
\Bigg|\sum_{P\in\mb H_{g}} \chi_{P}(f)\Bigg| \ll \deg(f) \frac{|P|^{\frac{1}2}}{\log_{q}|P|}.
\end{equation}

\begin{prop}\label{C-non-S-001}
\begin{enumerate}
\item
For $s \in \CC$ with ${\rm Re}(s)\ge \frac12$, we have
\begin{align}
\sum_{P\in\mb H_{g}}\sum_{n=0}^{g} (\pm 1)^{n} q^{-ns} \sum_{\substack{f\in\A_{n}^+\\ f\ne\square}} \chi_{P}(f)
= \begin{cases}
O(|P|^{1-\frac{s}2}) & \text{ if ${\rm Re}(s) < 1$}, \vspace{0.1em}\\
O(|P|^{\frac{1}2}(\log_{q}|P|)) & \text{ if ${\rm Re}(s) \ge 1$.}
\end{cases}
\end{align}
\item
For $s \in \CC$ with ${\rm Re}(s)\ge \frac12$, we have
\begin{align}
q^{(1-2s)g} \sum_{P\in\mb H_{g}}\sum_{n=0}^{g-1} (\pm 1)^{n} q^{(s-1)n} \sum_{\substack{f\in\A_{n}^+\\ f\ne\square}} \chi_{P}(f)
= O\left(|P|^{1-\frac{s}2}\right).
\end{align}
\item
Let $h\in\{g-1, g\}$.
For $s \in \CC$ with ${\rm Re}(s)\ge \frac12$, we have
\begin{align}
q^{-(h+1)s} \sum_{P\in\mb H_{g}}\sum_{n=0}^{h} \sum_{\substack{f\in\A_{n}^+\\ f\ne\square}} \chi_{P}(f)
= O\left(|P|^{1-\frac{s}2}\right).
\end{align}
\item
For $s \in \CC$ with ${\rm Re}(s)\ge \frac12$, we have
\begin{align}
\zeta_{\A}(2)^{-1} q^{-g} \sum_{P\in\mb P_{2g+2}}\sum_{n=0}^{g-1} (g-n) \sum_{\substack{f\in\A_{n}^+\\ f\ne\square}} \chi_{P}(f)
= O\left(|P|^{\frac{1}2}\right).
\end{align}
\end{enumerate}
\end{prop}
\begin{proof}
(1)
By \eqref{bound-2}, we have
\begin{align}
\sum_{P\in\mb H_{g}}\sum_{n=0}^{g} (\pm 1)^{n} q^{-ns} \sum_{\substack{f\in\A_{n}^+\\ f\ne\square}} \chi_{P}(f)
&\ll \sum_{n=0}^{g} q^{-ns} \sum_{\substack{f\in\A_{n}^+\\ f\ne\square}} \left|\sum_{P\in\mb H_{g}} \chi_{P}(f)\right| \\
&\ll \frac{|P|^{\frac{1}2}}{\log_{q}|P|} \sum_{n=0}^{g} n q^{(1-s)n}.
\end{align}
Since
\begin{align}
\sum_{n=0}^{g} n q^{(1-s)n} \ll \begin{cases}
g q^{(1-s)g} & \text{ if ${\rm Re}(s)<1$,} \\
g^2 & \text{ if ${\rm Re}(s) \ge 1$,}
\end{cases}
\end{align}
we have
\begin{align}
\frac{|P|^{\frac{1}2}}{\log_{q}|P|} \sum_{n=0}^{g} n q^{(1-s)n} \ll \begin{cases}
|P|^{1-\frac{s}2} & \text{ if ${\rm Re}(s) < 1$}, \\
|P|^{\frac{1}2}(\log_{q}|P|) & \text{ if ${\rm Re}(s) \ge 1$.}
\end{cases}
\end{align}
The proofs of (2), (3) and (4) are similar as that of (1).
\end{proof}

\subsubsection{\bf Proof of Theorem \ref{F-M-Theorem} (1)}\label{Odd-F-M-SS3}
By Lemma \ref{A-FE} (1), we have
\begin{align}
\sum_{P\in\mb P_{2g+1}} L(s,\chi_{P}) = \sum_{P\in\mb P_{2g+1}} \sum_{n=0}^{g} q^{-sn} \sum_{f\in\A_{n}^{+}} \chi_{P}(f)
+ q^{(1-2s)g} \sum_{P\in\mb P_{2g+1}} \sum_{n=0}^{g-1} q^{(s-1)n} \sum_{f\in\A_{n}^{+}} \chi_{P}(f).
\end{align}
We can write $\sum_{f\in\A_{n}^{+}} \chi_{P}(f)$ as
\begin{align}
\sum_{f\in\A_{n}^{+}} \chi_{P}(f)
= \sum_{\substack{f\in\A_{n}^{+}\\f=\square}} \chi_{P}(f) + \sum_{\substack{f\in\A_{n}^{+}\\f\ne\square}} \chi_{P}(f).
\end{align}
Then, by Propositions \ref{Contribution-Square-001} (1), (2) and \ref{C-non-S-001} (1), (2), we have
\begin{align}\label{F-M-R-004}
\sum_{P\in\mb P_{2g+1}} L(s,\chi_{P}) = \left(A_{g}(s) + B_{g}(s)\right) \frac{|P|}{\log_{q}|P|}
+ \begin{cases}
O(|P|^{1-\frac{s}2}) & \text{ if ${\rm Re}(s) < 1$}, \vspace{0.1em}\\
O(|P|^{\frac{1}2}(\log_{q}|P|)) & \text{ if ${\rm Re}(s) \ge 1$.}
\end{cases}
\end{align}
A simple computation shows that
\begin{align}\label{F-M-R-005}
A_{g}(s) + B_{g}(s) = \begin{cases}
g + 1 & \text{ if $s = \frac{1}2$,} \\
\zeta_{\A}(2s) (1- q^{(1+g)(1-2s)}) & \text{ if $s \ne \frac{1}2$.}
\end{cases}
\end{align}
For ${\rm Re(s)} \ge 1$, we have
\begin{align}\label{F-M-R-006}
\zeta_{\A}(2s) \left(1- q^{(1+g)(1-2s)}\right) \frac{|P|}{\log_{q}|P|}
= \zeta_{\A}(2s) \frac{|P|}{\log_{q}|P|} + O\left(|P|^{\frac{1}2}(\log_{q}|P|)\right).
\end{align}
Then, by \eqref{F-M-R-004}, \eqref{F-M-R-005} and \eqref{F-M-R-006}, we have
\begin{align}
\sum_{P\in\mb P_{2g+1}} L(s,\chi_{P}) = I_{g}(s) \frac{|P|}{\log_{q}|P|}
+ \begin{cases}
O(|P|^{1-\frac{s}2}) & \text{ if ${\rm Re}(s) < 1$}, \vspace{0.1em}\\
O(|P|^{\frac{1}2}(\log_{q}|P|)) & \text{ if ${\rm Re}(s) \ge 1$.}
\end{cases}
\end{align}
This completes the proof of Theorem \ref{F-M-Theorem} (1). \hfill\fbox

\subsubsection{\bf Proof of Theorem \ref{F-M-Theorem} (2)}\label{Odd-F-M-SS4}
By Lemma \ref{A-FE} (2), we can write
\begin{align}
\sum_{P\in\mb P_{2g+2}} L(s,\chi_{P})
&= \sum_{P\in\mb P_{2g+2}} \sum_{n=0}^{g} q^{-sn} \sum_{f\in\A_{n}^{+}} \chi_{P}(f)  \nonumber \\
&\hspace{2em}- q^{-(g+1)s}\sum_{P\in\mb P_{2g+2}} \sum_{n=0}^{g} \sum_{f\in\A_{n}^{+}} \chi_{P}(f) + \sum_{P\in\mb P_{2g+2}} H_{P}(s),
\end{align}
where
\begin{align}
H_{P}(1) = \zeta_{\A}(2)^{-1} q^{-g}\dis\sum_{n=0}^{g-1} \left(g-n\right) \sum_{f\in\A_{n}^{+}} \chi_{P}(f)
\end{align}
and, for $s\ne 1$,
\begin{align}
H_{P}(s) = \eta(s) q^{(1-2s)g} \sum_{n=0}^{g-1} q^{(s-1)n} \sum_{f\in\A_{n}^{+}} \chi_{P}(f)
- \eta(s) q^{-s g} \sum_{n=0}^{g-1} \sum_{f\in\A_{n}^{+}} \chi_{P}(f)
\end{align}
with $\eta(s) = \frac{\zeta_{\A}(2-s)}{\zeta_{\A}(1+s)}$.
We first consider the case $s=1$.
By Propositions \ref{Contribution-Square-001} (1), (2), (4) and \ref{C-non-S-001} (1), (2), (4), we have
\begin{align}\label{F-M-Real-005}
\sum_{P\in\mb P_{2g+2}} L(1,\chi_{P}) = H(g) \frac{|P|}{\log_{q}|P|} + O\left(|P|^{\frac{1}2}(\log_{q}|P|)\right),
\end{align}
where $H(g) = A_{g}(1) - C_{g}(1) + B(g) |P|^{-\frac{1}2}$.
Since
\begin{align}
C_{g}(1) \frac{|P|}{\log_{q}|P|}
= \zeta_{\A}(2) q^{[\frac{g}2]-(g+1)} \frac{|P|}{\log_{q}|P|} + O\left(|P|^{\frac{1}2}(\log_{q}|P|)\right)
\end{align}
and
\begin{align}
B(g) \frac{|P|^{\frac{1}2}}{\log_{q}|P|}
= q^{([\frac{g-1}2]+1)} \left\{2\zeta_{\A}(2)- \tfrac{1}{2}(1+(-1)^{g+1})\right\} \frac{|P|^{\frac{1}2}}{\log_{q}|P|}
+ O\left(|P|^{\frac{1}2}(\log_{q}|P|)\right),
\end{align}
we have
\begin{equation}\label{F-M-Real-007}
H(g) \frac{|P|}{\log_{q}|P|} = J_{g}(1) \frac{|P|}{\log_{q}|P|} + O\left(|P|^{\frac{1}2}(\log_{q}|P|)\right),
\end{equation}
where
\begin{align}
J_{g}(1) := A_{g}(1) - \zeta_{\A}(2) q^{[\frac{g}2]-(g+1)} + q^{([\frac{g-1}2]-g)} \left\{2\zeta_{\A}(2)- \tfrac{1}{2}(1+(-1)^{g+1})\right\}.
\end{align}
A simple computation shows that $J_{g}(1) = \zeta_{\A}(2)$.
Hence, by \eqref{F-M-Real-005} and \eqref{F-M-Real-007}, we have
\begin{align}
\sum_{P\in\mb P_{2g+2}} L(1,\chi_{P})  = \zeta_{\A}(2) \frac{|P|}{\log_{q}|P|} + O\left(|P|^{\frac{1}2}(\log_{q}|P|)\right).
\end{align}

Now, consider the case $s\ne 1$.
For $|s-1|>\varepsilon$, $\eta(s)$ is bounded.
Then, by Propositions \ref{Contribution-Square-001} and \ref{C-non-S-001}, we have
\begin{align}\label{F-M-Real-010}
\sum_{P\in\mb P_{2g+2}} L(s,\chi_{P}) = H_{g}(s) \frac{|P|}{\log_{q}|P|}
+ \begin{cases}
O(|P|^{1-\frac{s}2}) & \text{ if ${\rm Re}(s) < 1$}, \vspace{0.1em} \\
O(|P|^{\frac{1}2}(\log_{q}|P|)) & \text{ if ${\rm Re}(s) \ge 1$,}
\end{cases}
\end{align}
where $H_{g}(s) = A_{g}(s) + \eta(s) B_{g}(s) - C_{g}(s) - \eta(s) C_{g-1}(s)$.
We have that, for $s\in\mb C$ with $\frac{1}2 \le {\rm Re}(s) < 1$ and $|s-1|>\varepsilon$,
\begin{align}\label{F-M-Real-011}
C_{g}(s) \frac{|P|}{\log_{q}|P|}
&= \zeta_{\A}(2) q^{[\frac{g}2]-(g+1)s} \frac{|P|}{\log_{q}|P|} + O\left(|P|^{1-\frac{s}2}\right),  \\
\eta(s) C_{g-1}(s) \frac{|P|}{\log_{q}|P|}
&= \zeta_{\A}(2) \eta(s) q^{[\frac{g-1}2]-gs} \frac{|P|}{\log_{q}|P|} + O\left(|P|^{1-\frac{s}2}\right),
\end{align}
and, for $s\in\mb C$ with $1 \le {\rm Re}(s) < \frac{3}2$ and $|s-1|>\varepsilon$,
\begin{align}\label{F-M-Real-012}
\eta(s) B_{g}(s) \frac{|P|}{\log_{q}|P|} &= \eta(s) \zeta_{\A}(2s) q^{(1-2s)(g-[\frac{g-1}2])} \frac{|P|}{\log_{q}|P|}
+ O\left(|P|^{\frac{1}2}(\log_{q}|P|)\right),  \\
C_{g}(s) \frac{|P|}{\log_{q}|P|}
&= \zeta_{\A}(2) q^{[\frac{g}2]-(g+1)s} \frac{|P|}{\log_{q}|P|} + O\left(|P|^{\frac{1}2}(\log_{q}|P|)\right),  \\
\eta(s) C_{g-1}(s) \frac{|P|}{\log_{q}|P|}
&= \zeta_{\A}(2) \eta(s) q^{[\frac{g-1}2]-gs} \frac{|P|}{\log_{q}|P|} + O\left(|P|^{\frac{1}2}(\log_{q}|P|)\right),
\end{align}
and, for $s\in\mb C$ with $\frac{3}2 \le {\rm Re}(s)$,
\begin{align}\label{F-M-Real-013}
A_{g}(s) \frac{|P|}{\log_{q}|P|} &= \zeta_{\A}(2s) \frac{|P|}{\log_{q}|P|} + O\left(|P|^{\frac{1}2}(\log_{q}|P|)\right),  \\
\eta(s) B_{g}(s) \frac{|P|}{\log_{q}|P|} &\ll |P|^{\frac{1}2} (\log_{q}|P|),  \\
C_{g}(s) \frac{|P|}{\log_{q}|P|} &\ll |P|^{\frac{1}2}(\log_{q}|P|),  \\
\eta(s) C_{g-1}(s) \frac{|P|}{\log_{q}|P|} &\ll |P|^{\frac{1}2}(\log_{q}|P|).
\end{align}
Then, by \eqref{F-M-Real-011}, \eqref{F-M-Real-012} and \eqref{F-M-Real-013}, we have
\begin{align}\label{F-M-Real-014}
H_{g}(s) \frac{|P|}{\log_{q}|P|} = J_{g}(s) \frac{|P|}{\log_{q}|P|} + \begin{cases}
O(|P|^{1-\frac{s}2}) & \text{ if ${\rm Re}(s) < 1$}, \vspace{0.1em} \\
O(|P|^{\frac{1}2}(\log_{q}|P|)) & \text{ if ${\rm Re}(s) \ge 1$.}
\end{cases}
\end{align}
Hence, by \eqref{F-M-Real-010} and \eqref{F-M-Real-014}, we have
\begin{align}
\sum_{P\in\mb P_{2g+2}} L(s,\chi_{P}) = J_{g}(s) \frac{|P|}{\log_{q}|P|}
+ \begin{cases}
O(|P|^{1-\frac{s}2}), & \text{ if ${\rm Re}(s) < 1$} \vspace{0.1em} \\
O(|P|^{\frac{1}2}(\log_{q}|P|)) & \text{ if ${\rm Re}(s) \ge 1$}
\end{cases}
\end{align}
for $s\in\mb C$ with $\frac{1}2 \le {\rm Re}(s)$ and $|s-1|>\varepsilon$.
This completes the proof of Theorem \ref{F-M-Theorem} (2). \hfill\fbox

\subsubsection{\bf Proof of Theorem \ref{F-M-Theorem} (3)}\label{Odd-F-M-SS5}
For any $f\in\A^{+}$, we have $\chi_{\gamma P}(f) = (-1)^{\deg(f)} \chi_{P}(f)$.
Hence, by Lemma \ref{A-FE} (3), we have
\begin{align}
\sum_{P\in\mb P_{2g+2}} L(s,\chi_{\gamma P})
&= \sum_{P\in\mb P_{2g+2}} \sum_{n=0}^{g} (-1)^{n} q^{-sn} \sum_{f\in\A_{n}^{+}} \chi_{P}(f)
+ (-1)^{g} q^{-(g+1)s} \sum_{P\in\mb P_{2g+2}} \sum_{n=0}^{g}\sum_{f\in\A_{n}^{+}} \chi_{P}(f) \nonumber \\
&\quad + \nu(s) q^{(1-2s)g} \sum_{P\in\mb P_{2g+2}} \sum_{n=0}^{g-1} (-1)^{n} q^{(s-1)n} \sum_{f\in\A_{n}^{+}} \chi_{P}(f)  \nonumber \\
&\quad + (-1)^{g+1}\nu(s) q^{-sg} \sum_{P\in\mb P_{2g+2}} \sum_{n=0}^{g-1} \sum_{f\in\A_{n}^{+}} \chi_{P}(f),
\end{align}
where $\nu(s) = \frac{1+q^{-s}}{1+q^{s-1}}$.
Following the same process as in the proof of Theorem \ref{F-M-Theorem} (2), we can show that
\begin{align}
\sum_{P\in\mb P_{2g+2}} L(s,\chi_{\gamma P}) = K_{g}(s) \frac{|P|}{\log_{q}|P|} + \begin{cases}
O(|P|^{1-\frac{s}2}) & \text{ if ${\rm Re}(s) < 1$}, \vspace{0.1em}\\
O(|P|^{\frac{1}2} (\log_{q}|P|)) & \text{ if ${\rm Re}(s) \ge 1$.}
\end{cases}
\end{align}
This completes the proof of Theorem \ref{F-M-Theorem} (3). \hfill\fbox

\subsection{Even characteristic case}
In this subsection, we give a proof of Theorem \ref{E-F-M-Theorem}.
In \S\ref{Even-F-M-SS1}, we obtain several results of the contribution of squares and of non-squares,
which will be used to calculate the first moment of $L$-functions
in \S\ref{Even-F-M-SS3}, \S\ref{Even-F-M-SS4} and \S\ref{Even-F-M-SS5}.

\subsubsection{\bf Auxiliary Lemmas}
\begin{lem}\label{E-PNT}
\begin{enumerate}
\item
For any positive integer $n$, we have
\begin{align}\label{E-PNT-2-1}
\sum_{P\in\mb P_{n}}\sum_{u\in\mc F_{P}} 1 = \frac{|P|^{2}}{\log_{q}|P|} + O\left(\frac{|P|^{\frac{3}{2}}}{\log_{q}|P|}\right).
\end{align}
\item
For any positive integer $n$, we have
\begin{align}\label{E-PNT-2-2}
\#\mc H_{n+1} = 2 \frac{q^{2n+1}}{n}+ O\left(\frac{q^{2n}}{n^2}\right).
\end{align}
\end{enumerate}
\end{lem}
\begin{proof}
(1) By Theorem \ref{thm:pnt}, we have
\begin{align}
\sum_{P\in\mb P_{n}}\sum_{u\in\mc F_{P}} 1 = \sum_{P\in\mb P_{n}}\sum_{\substack{0\ne A \in \A^{+} \\ \deg(A) < n}} 1
&= (|P|-1) \frac{|P|}{\log_{q}|P|} + O\left(\frac{|P|^{\frac{3}2}}{\log_{q}|P|}\right) \\
&= \frac{|P|^{2}}{\log_{q}|P|} + O\left(\frac{|P|^{\frac{3}{2}}}{\log_{q}|P|}\right).
\end{align}
(2) Since $\#\mc G_{n+1-r} = 2 \zeta_{\A}(2)^{-1} q^{n+1-r}$, by \eqref{E-PNT-2-1}, we have
\begin{align}
\#\mc H_{(r,n+1-r)} = \#\mc F_{r} \cdot \#\mc G_{n+1-r} = 2 \zeta_{\A}(2)^{-1} q^{n+1-r} (q^{r}-1) \#\mb P_{r}
\end{align}
and
\begin{align}
\#\mc H_{n+1} = \sum_{r=1}^{n} \#\mc H_{(r,n+1-r)} = 2 \zeta_{\A}(2)^{-1} q^{n+1} \sum_{r=1}^{n} (1-q^{-r}) \#\mb P_{r}.
\end{align}
From \cite[Theorem 2]{Po10}, we can deduce that
\begin{align}
\sum_{r=1}^{n} \#\mb P_{r} = \frac{q}{q-1} \frac{q^{n}}{n} + O\left(\frac{q^{n}}{n^2}\right).
\end{align}
Also, since $\#\mb P_{r} \le \frac{q^r}{r}$, we have
\begin{align}
\sum_{r=1}^{n} q^{-r} \#\mb P_{r} \le \sum_{r=1}^{n} \frac{1}{r} \ll \log n.
\end{align}
Hence, we get the result.
\end{proof}

For $P\in\mb P$ and $f\in\A^{+}$, let $\Gamma_{f, P}$ and $T_{f,P}$ be defined by (see \cite[\S3]{Ch08})
\begin{align}
\Gamma_{f, P} = \sum_{\substack{A\in\A \\ \deg(A) < \deg(P)}} \left\{\frac{A/P}{f}\right\} \quad\text{ and }\quad
T_{f, P} = \sum_{\substack{0 \ne A\in\A \\ \deg(A) < \deg(P)}} \left\{\frac{A/P}{f}\right\}.
\end{align}

\begin{lem}\label{E-BOUND-2-1}
For $f\in\A^{+}$ with $\deg(f) \le 2g+1$, which is not a perfect square, we have
\begin{align}
\left|\sum_{P\in\mb P_{g+1}}\sum_{u\in\mc F_{P}} \chi_{u}(f) \right| \ll \frac{|P|}{\log_{q}|P|}.
\end{align}
\end{lem}
\begin{proof}
By \cite[Lemma 3.1]{Ch08}, we have $\Gamma_{f, P} = 0$, so $T_{f,P} = \Gamma_{f, P} - \{\frac{0}{f}\} = -1$.
Then, we have
\begin{align}
\sum_{P\in\mb P_{g+1}}\sum_{u\in\mc F_{P}} \chi_{u}(f)
&= \sum_{P\in\mb P_{g+1}} \sum_{\substack{0 \ne A \in \A \\ \deg(A) < g+1}} \left\{\frac{A/P}{f}\right\} \\
&= \sum_{P\in\mb P_{g+1}} T_{f, P} = - \sum_{\substack{P\in\mb P_{g+1}\\P\nmid f}} 1 \ll \#\mb P_{g+1}.
\end{align}
Hence, by Theorem \ref{thm:pnt}, we get the result.
\end{proof}

For $P\in\mb P, f\in\A^{+}$ and positive integer $s$, let $\Gamma_{f,P,s}$ and $T_{f,P,s}$ be defined by
\begin{align}
\Gamma_{f,P,s} = \sum_{\substack{A\in\A \\ \deg(A) < \deg(P)}} \sum_{F\in\mc G_{s}} \left\{\frac{A/P+F}{f}\right\} ~~\text{ and }~~
T_{f,P,s} = \sum_{\substack{0 \ne A\in\A \\ \deg(A) < \deg(P)}} \sum_{F\in\mc G_{s}} \left\{\frac{A/P+F}{f}\right\}.
\end{align}

\begin{lem}\label{E-BOUND-3}
Let $P\in\mb P, f\in\A^{+}$ and $s$ be a positive integer with $\deg(f) \le 2 \deg(P) + 2s - 2$.
Suppose that $P\nmid f$ and $f$ is not a perfect square.
Then ${\Gamma}_{f, P, s} = 0$ and $T_{f,P,s} \ll q^{s}$.
\end{lem}
\begin{proof}
Let $\mc E_{P, s}$ be the set of rational functions $u = \frac{A}{P} + F \in k$ such that $\deg(A) < \deg(P)$
and $F = \sum_{n=1}^{s-1} \alpha_{n} T^{2n-1}$ with $\alpha_{n} \in \bF_{q}$.
Then, we have
\begin{align}\label{E-BOUND-3-001}
\Gamma_{f,P,s} = \left(1+\left\{\frac{\gamma}{f}\right\}\right)
\left(\sum_{\alpha\in\bF_{q}^*} \left\{\frac{\alpha T^{2s-1}}{f}\right\} \sum_{u\in\mc E_{P,s}} \left\{\frac{u}{f}\right\}\right),
\end{align}
where $\gamma$ is a generator of $\bF_{q}^{\times}$.
Since $\mc E_{P, s}$ is abelian group and $\{\frac{\cdot}{f}\}$ is a homomorphism, we have
\begin{align}\label{E-BOUND-3-002}
\sum_{u\in\mc E_{P,s}} \left\{\frac{u}{f}\right\} = 0.
\end{align}
From \eqref{E-BOUND-3-001} and \eqref{E-BOUND-3-002}, we have $\Gamma_{f,P,s} = 0$ and
$T_{f,P,s} = - \sum_{F\in\mc G_{s}} \{\frac{F}{f}\} \ll \#\mc G_{s} \ll q^{s}$.
\end{proof}

\begin{lem}\label{E-BOUND-3-1}
For $f\in\A^{+}$ with $\deg(f)< 2g+1$, which is not a perfect square, we have
\begin{align}
\left|\sum_{u\in\mc H_{g+1}} \chi_{u}(f)\right| \ll (\log g) q^{g}.
\end{align}
\end{lem}
\begin{proof}
By Lemma \ref{E-BOUND-3}, we have
\begin{align}
\sum_{u\in\mc H_{(r,g+1-r)}} \chi_{u}(f) = \sum_{P\in\mb P_{r}} T_{f, P, g+1-r}
\ll q^{g+1-r} \#\mb P_{r} \ll \frac{q^{g}}{r}.
\end{align}
Hence, we have
\begin{align}
\left|\sum_{u\in\mc H_{g+1}} \chi_{u}(f)\right| = \left|\sum_{r=1}^{g}\sum_{u\in\mc H_{(r,g+1-r)}} \chi_{u}(f)\right|
\ll q^{g} \sum_{r=1}^{g} \frac{1}{r} \ll (\log g) q^{g}.
\end{align}
\end{proof}

\subsubsection{\bf Preparations for the proof}\label{Even-F-M-SS1}
We first consider the contributions of squares.

\begin{prop}\label{E-Contribution-Square-Ramified-001}
\begin{enumerate}
\item
For $s=\frac{1}{2}$, we have
\begin{align}
\sum_{u\in\mc H_{g+1}} \sum_{n=0}^{g} q^{-\frac{n}2} \sum_{\substack{f\in\A_{n}^{+} \\ f = \square}} \chi_{u}(f)
= 2([\tfrac{g}2]+1) \frac{q^{2g+1}}{g} + O\left(\frac{q^{2g}}{g}\right)
\end{align}
and, for $s\in\mb C$ with ${\rm Re}(s) \ge \frac{1}2 ~(s \ne \frac{1}2)$,
\begin{align}
\sum_{u\in\mc H_{g+1}} \sum_{n=0}^{g} q^{-sn} \sum_{\substack{f\in\A_{n}^{+} \\ f = \square}} \chi_{u}(f)
= 2 \zeta_{\A}(2s) \frac{q^{2g+1}}{g} + O\left(\frac{q^{2g}}{g^2}\right).
\end{align}
\item
For $s=\frac{1}{2}$, we have
\begin{align}
\sum_{u\in\mc H_{g+1}} \sum_{n=0}^{g-1} q^{-\frac{n}2} \sum_{\substack{f\in\A_{n}^{+} \\ f = \square}} \chi_{u}(f)
= 2 ([\tfrac{g-1}2]+1) \frac{q^{2g+1}}{g} + O\left(\frac{q^{2g}}{g}\right)
\end{align}
and, for $s\in\mb C$ with ${\rm Re}(s) \ge \frac{1}2 ~(s \ne \frac{1}2)$,
\begin{align}
\sum_{u\in\mc H_{g+1}} \sum_{n=0}^{g-1} q^{(s-1)n} \sum_{\substack{f\in\A_{n}^{+} \\ f = \square}} \chi_{u}(f) = O\left(\frac{q^{2g}}{g^2}\right).
\end{align}
\end{enumerate}
\end{prop}
\begin{proof}
(1) We have
\begin{align}
\sum_{u\in\mc H_{g+1}} \sum_{n=0}^{g} q^{-sn} \sum_{\substack{f\in\A_{n}^{+} \\ f = \square}} \chi_{u}(f)
&=\sum_{r=1}^{g} \sum_{u\in\mc H_{(r,g+1-r)}} \sum_{n=0}^{g} q^{-sn} \sum_{\substack{f\in\A_{n}^{+} \\ f = \square}} \chi_{u}(f) \\
&= \sum_{r=1}^{g} \sum_{u\in\mc H_{(r,g+1-r)}} \sum_{l=0}^{[\frac{g}2]} q^{-2sl} \sum_{L\in\A_{l}^{+}} \chi_{u}(L^2).
\end{align}
For any $u = v+F\in\mc H_{(r,g+1-r)}$ with $v=\frac{A}{P}\in\mc F_{r}$ and $F\in\mc G_{g+1-r}$, we have
\begin{align}
\sum_{L\in\A_{l}^{+}} \chi_{u}(L^2) = \sum_{\substack{L\in\A_{l}^{+}\\ (L,P)=1}} 1
= \begin{cases}
q^{l} & \text{ if $l < r$,} \\ q^{l} - q^{l-r} & \text{ if $l \ge r$}.
\end{cases}
\end{align}
Then we have
\begin{align}
\sum_{r=1}^{g} \sum_{u\in\mc H_{(r,g+1-r)}} \sum_{n=0}^{g} q^{-sn} \sum_{\substack{f\in\A_{n}^{+} \\ f = \square}} \chi_{u}(f)
= {A}_{g}(s) \sum_{r=1}^{g} \#\mc H_{(r,g+1-r)} - \alpha_{g}(s),
\end{align}
where $A_{g}(s)$ is given in Proposition \ref{Contribution-Square-001} (1) and
\begin{align}
\alpha_{g}(s) = \sum_{r=1}^{[\frac{g}2]} q^{-r} \#\mc H_{(r,g+1-r)} \sum_{l=r}^{[\frac{g}2]} q^{(1-2s)l}.
\end{align}
By \eqref{E-PNT-2-2}, we have
\begin{align}
{A}_{g}(s) \sum_{r=1}^{g} \#\mc H_{(r,g+1-r)}
= 2 {A}_{g}(s) \frac{q^{2g+1}}{g} + O\left(A_{g}(s) \frac{q^{2g}}{g^2}\right).
\end{align}
For $s=\frac{1}2$, we have $A_{g}(\frac{1}2) \frac{q^{2g}}{g^2} \ll \frac{q^{2g}}{g}$.
For $s\ne \frac{1}2$, we have $A_{g}(s) \frac{q^{2g}}{g^2} \ll \frac{q^{2g}}{g^2}$ and
\begin{align}
{A}_{g}(s) \frac{q^{2g+1}}{g} = \zeta_{\A}(2s) \frac{q^{2g+1}}{g} + O\left(\frac{q^{2g}}{g^2}\right).
\end{align}
Since $\#\mc H_{(r,g+1-r)} \ll q^{g} \frac{q^r}{r}$,
\begin{align}
\alpha_{g}(s) \ll q^{g} \sum_{r=1}^{[\frac{g}2]} \frac{1}{r} \sum_{l=r}^{[\frac{g}2]} q^{(1-2s)l} \ll (\log g) g q^{g} \ll \frac{q^{2g}}{g^2}.
\end{align}
Therefore, we get the result.
Similarly, we can prove (2).
\end{proof}

\begin{prop}\label{E-Contribution-Square-001}

\begin{enumerate}
\item
For $s\in\mb C$ with ${\rm Re}(s) \ge \frac{1}2$, we have
\begin{align}
\hspace{2em}\sum_{P\in\mb P_{g+1}}\sum_{u\in\mc F_{P}} \sum_{n=0}^{g} (\pm 1)^{n} q^{-sn} \sum_{\substack{f\in\A_{n}^+\\ f=\square}} \chi_{u}(f)
= \tilde{A}_{g}(s) \frac{|P|^{2}}{\log_{q}|P|} + O\left(|P|^{\frac{3}{2}}\right),
\end{align}
where 
\begin{align}
\tilde{A}_{g}(s) = \begin{cases}
[\frac{g}{2}]+1 & \text{ if } s = \frac{1}{2}, \\
\zeta_{\A}(2s)(1-q^{(1-2s)([\frac{g}2]+1)}) & \text{ if } {\rm Re}(s) < 1 (s\ne\frac{1}{2}), \\
\zeta_{\A}(2s) & \text{ if } {\rm Re}(s) \ge 1.
\end{cases}
\end{align}
\item
For $s\in\mb C$ with ${\rm Re}(s) \ge \frac{1}2$, we have
\begin{align}
\hspace{3em}q^{(1-2s)g} \sum_{P\in\mb P_{g+1}}\sum_{u\in\mc F_{P}} \sum_{n=0}^{g-1} (\pm 1)^{n} q^{(s-1)n} \sum_{\substack{f\in\A_{n}^+\\ f=\square}} \chi_{u}(f)
= \tilde{B}_{g}(s) \frac{|P|^{2}}{\log_{q}|P|} + O\left(|P|^{2-s}\right),
\end{align}
where 
\begin{align}
\tilde{B}_{g}(s) = \begin{cases}
[\frac{g-1}{2}] + 1 & \text{ if } s = \frac{1}{2}, \\
\zeta_{\A}(2s) q^{(g-[\frac{g-1}{2}])(1-2s)} & \text{ if } s \ne \frac{1}{2}.
\end{cases}
\end{align}
\item
Let $h \in \{g-1, g\}$.
For $s\in\mb C$ with ${\rm Re}(s) \ge \frac{1}2$, we have
\begin{align}
\hspace{3em}q^{-(h+1)s} \sum_{P\in\mb P_{g+1}}\sum_{u\in\mc F_{P}} \sum_{n=0}^{h} \sum_{\substack{f\in\A_{n}^+\\ f=\square}} \chi_{u}(f)
= \zeta_{\A}(2) q^{[\frac{h}2]-(h+1)s} \frac{|P|^{2}}{\log_{q}|P|} + O\left(|P|^{2-s}\right).
\end{align}
\item
For $s\in\mb C$ with ${\rm Re}(s) \ge \frac{1}2$, we have
\begin{align}
\hspace{3em}\zeta_{\A}(2)^{-1} q^{-g} \sum_{P\in\mb P_{g+1}}\sum_{u\in\mc F_{P}} \sum_{n=0}^{g-1} (g-n) \sum_{\substack{f\in\A_{n}^+\\ f=\square}} \chi_{u}(f)
= \tilde{B}(g) \frac{|P|}{\log_{q}|P|} + O\left(|P|\right),
\end{align}
where
\begin{align}
\tilde{B}(g) = q^{([\frac{g-1}2]+1)} \left\{2\zeta_{\A}(2) - \tfrac{1}2 (1+(-1)^{g+1})\right\}.
\end{align}
\end{enumerate}
\end{prop}
\begin{proof}
(1) By \eqref{E-PNT-2-1}, we have
\begin{align}\label{E-Contribution-Square-001-1}
\sum_{P\in\mb P_{g+1}}\sum_{u\in\mc F_{P}} \sum_{n=0}^{g} (\pm 1)^{n} q^{-sn} \sum_{\substack{f\in\A_{n}^+\\ f=\square}} \chi_{u}(f)
&= \sum_{l=0}^{[\frac{g}2]} q^{-2sl} \sum_{L\in\A_{l}^+} \sum_{P\in\mb P_{g+1}}\sum_{u\in\mc F_{P}} \chi_{u}(L^2) \nonumber \\
&= \sum_{l=0}^{[\frac{g}2]} q^{(1-2s)l} \sum_{P\in\mb P_{g+1}} \sum_{u\in\mc F_{P}} 1\nonumber \\
&= {A}_{g}(s) \frac{|P|^{2}}{\log_{q}|P|} + O\left({A}_{g}(s) \frac{|P|^{\frac{3}{2}}}{\log_{q}|P|}\right),
\end{align}
where $A_{g}(s)$ is given in Proposition \ref{Contribution-Square-001} (1).
Since ${A}_{g}(s) \ll g+1$, the error term in \eqref{E-Contribution-Square-001-1} is $\ll |P|^{\frac{3}{2}}$.
Since $q^{(1-2s)([\frac{g}2]+1)} |P|^2 \ll |P|^{\frac{3}{2}}$ for ${\rm Re}(s) \ge 1$, we have
\begin{align}
A_{g}(s) \frac{|P|^{2}}{\log_{q}|P|} = \tilde{A}_{g}(s) \frac{|P|^{2}}{\log_{q}|P|} + O\left(|P|^{\frac{3}{2}}\right).
\end{align}
Hence, we get the result.
Similarly, we can prove (2), (3) and (4).
\end{proof}

Now we consider the contribution of non-squares.

\begin{prop}\label{E-C-non-S-Ramified-001}
\begin{enumerate}
\item
For $s\in\mb C$ with ${\rm Re}(s) \ge \frac{1}2$, we have
\begin{align}
\hspace{2em}\sum_{u\in\mc H_{g+1}} \sum_{n=0}^{g} q^{-sn} \sum_{\substack{f\in\A_{n}^{+} \\ f \ne \square}} \chi_{u}(f)
= O\left((\log g) g q^{\frac{3g}{2}}\right).
\end{align}
\item
For $s\in\mb C$ with ${\rm Re}(s) \ge \frac{1}2$, we have
\begin{align}
q^{(1-2s)g} \sum_{u\in\mc H_{g+1}} \sum_{n=0}^{g-1} q^{(s-1)n} \sum_{\substack{f\in\A_{n}^{+} \\ f \ne \square}} \chi_{u}(f)
= O\left((\log g) g q^{\frac{3g}{2}}\right).
\end{align}
\end{enumerate}
\end{prop}
\begin{proof}
(1) By using the fact that $\sum_{r=1}^{g} \frac{1}{r} \ll \log g$ and Lemma \ref{E-BOUND-3-1}, we have
\begin{align}
\sum_{u\in\mc H_{g+1}} \sum_{n=0}^{g} q^{-sn} \sum_{\substack{f\in\A_{n}^{+} \\ f \ne \square}} \chi_{u}(f)
&\ll \sum_{n=0}^{g} q^{-sn} \sum_{\substack{f\in\A_{n}^{+} \\ f \ne \square}} \left|\sum_{u\in\mc H_{g+1}} \chi_{u}(f)\right| \\
&\ll (\log g) q^{g} \sum_{n=0}^{g} q^{(1-s)n} \ll (\log g) g q^{\frac{3g}{2}}.
\end{align}
Similarly, we can prove (2).
\end{proof}

\begin{prop}\label{E-C-non-S-001}
\begin{enumerate}
\item
For $s \in \CC$ with ${\rm Re}(s)\ge \frac12$, we have
\begin{align}
\sum_{P\in\mb P_{g+1}}\sum_{u\in\mc F_{P}} \sum_{n=0}^{g} (\pm 1)^{n} q^{-sn} \sum_{\substack{f\in\A_{n}^+\\ f\ne\square}} \chi_{u}(f)
= \begin{cases}
O(|P|^{2-s}) & \text{ if ${\rm Re}(s) < 1$,} \vspace{0.1em}\\ O(|P|) & \text{ if ${\rm Re}(s) \ge 1$.}
\end{cases}
\end{align}
\item
For $s \in \CC$ with ${\rm Re}(s)\ge \frac12$, we have
\begin{align}
q^{(1-2s)g} \sum_{P\in\mb P_{g+1}}\sum_{u\in\mc F_{P}} \sum_{n=0}^{g-1} (\pm 1)^{n} q^{(s-1)n} \sum_{\substack{f\in\A_{n}^+\\ f\ne\square}} \chi_{u}(f)
= O\left(|P|^{2-s}\right).
\end{align}
\item
Let $h \in \{g-1, g\}$.
For $s \in \CC$ with ${\rm Re}(s)\ge \frac12$, we have
\begin{align}
q^{-(h+1)s} \sum_{P\in\mb P_{g+1}}\sum_{u\in\mc F_{P}} \sum_{n=0}^{h} \sum_{\substack{f\in\A_{n}^+\\ f\ne\square}} \chi_{u}(f)
= O\left(|P|^{2-s}\right).
\end{align}
\item
For $s \in \CC$ with ${\rm Re}(s)\ge \frac12$, we have
\begin{align}
\zeta_{\A}(2)^{-1} q^{-g} \sum_{P\in\mb P_{g+1}}\sum_{u\in\mc F_{P}} \sum_{n=0}^{g-1} (g-n) \sum_{\substack{f\in\A_{n}^+\\ f\ne\square}} \chi_{u}(f)
= O\left(|P|\right).
\end{align}
\end{enumerate}
\end{prop}
\begin{proof}
(1) By Lemma \ref{E-BOUND-2-1}, we have
\begin{align}
\sum_{P\in\mb P_{g+1}}\sum_{u\in\mc F_{P}} \sum_{n=0}^{g} (\pm 1)^{n} q^{-sn} \sum_{\substack{f\in\A_{n}^+\\ f\ne\square}} \chi_{u}(f)
&\ll \sum_{n=0}^{g} q^{-ns} \sum_{\substack{f\in\A_{n}^+\\ f\ne\square}} \left|\sum_{P\in\mb P_{g+1}}\sum_{u\in\mc F_{P}} \chi_{u}(f)\right| \\
&\ll \frac{|P|}{\log_{q}|P|} \sum_{n=0}^{g}q^{(1-s)n}.
\end{align}
Since
\begin{align}
\sum_{n=0}^{g} q^{(1-s)n} \ll \begin{cases}
|P|^{1-s} (\log_{q}|P|) & \text{ if ${\rm Re}(s)<1$,}  \\
\log_{q}|P| & \text{ if ${\rm Re}(s) \ge 1$,}
\end{cases}
\end{align}
we have
\begin{align}
\frac{|P|}{\log_{q}|P|} \sum_{n=0}^{g}q^{(1-s)n} \ll \begin{cases}
|P|^{2-s} & \text{ if ${\rm Re}(s) < 1$}, \\ |P| & \text{ if ${\rm Re}(s) \ge 1$.}
\end{cases}
\end{align}
Similarly, we can prove (2), (3) and (4).
\end{proof}

\subsubsection{\bf Proof of Theorem \ref{E-F-M-Theorem} (1)}\label{Even-F-M-SS3}
By Lemma \ref{A-FE} (1), we have
\begin{align}
\sum_{u\in\mc H_{g+1}} L(s,\chi_{u})
&= \sum_{u\in\mc H_{g+1}} \sum_{n=0}^{g} q^{-sn} \sum_{f\in\A_{n}^{+}} \chi_{u}(f)
+ q^{(1-2s)g}\sum_{u\in\mc H_{g+1}} \sum_{n=0}^{g} q^{(s-1)n} \sum_{f\in\A_{n}^{+}} \chi_{u}(f).
\end{align}
By Propositions \ref{E-Contribution-Square-Ramified-001} and \ref{E-C-non-S-Ramified-001}, for $s=\frac{1}{2}$, we have
\begin{align}
\sum_{u\in\mc H_{g+1}} L(\tfrac{1}2,\chi_{u}) = 2 \left([\tfrac{g}2]+[\tfrac{g-1}2]+2\right) \frac{q^{2g+1}}{g} + O\left(\frac{q^{2g}}{g}\right),
\end{align}
and, for $s\in\mb C$ with ${\rm Re}(s) \ge \frac{1}2 ~(s \ne \frac{1}2)$,
\begin{align}
\sum_{u\in\mc H_{g+1}} L(s,\chi_{u}) = 2 \zeta_{\A}(2s) \frac{q^{2g+1}}{g} + O\left(\frac{q^{2g}}{g^2}\right).
\end{align}
A simple computation shows that $[\frac{g}2]+[\frac{g-1}2]+2 = g+1$.
This completes the proof of Theorem \ref{E-F-M-Theorem} (1). \hfill\fbox

\subsubsection{\bf Proof of Theorem \ref{E-F-M-Theorem} (2)}\label{Even-F-M-SS4}
By Lemma \ref{A-FE} (2), we have
\begin{align}
\sum_{P\in\mb P_{g+1}}\sum_{u\in\mc F_{P}} L(s,\chi_{u})
&= \sum_{P\in\mb P_{g+1}}\sum_{u\in\mc F_{P}} \sum_{n=0}^{g} q^{-sn} \sum_{f\in\A_{n}^{+}} \chi_{u}(f)  \nonumber \\
&\hspace{-1em}- q^{-(g+1)s}\sum_{P\in\mb P_{g+1}}\sum_{u\in\mc F_{P}} \sum_{n=0}^{g} \sum_{f\in\A_{n}^{+}} \chi_{u}(f)
+ \sum_{P\in\mb P_{g+1}}\sum_{u\in\mc F_{P}} H_{u}(s),
\end{align}
where
\begin{align}
H_{u}(1) = \zeta_{\A}(2)^{-1} q^{-g}\dis\sum_{n=0}^{g-1} \left(g-n\right) \sum_{f\in\A_{n}^{+}} \chi_{u}(f)
\end{align}
and, for $s\ne 1$,
\begin{align}
H_{u}(s) = \eta(s) q^{(1-2s)g} \sum_{n=0}^{g-1} q^{(s-1)n} \sum_{f\in\A_{n}^{+}} \chi_{u}(f)
- \eta(s) q^{-s g} \sum_{n=0}^{g-1} \sum_{f\in\A_{n}^{+}} \chi_{u}(f)
\end{align}
with $\eta(s) = \frac{\zeta_{\A}(2-s)}{\zeta_{\A}(1+s)}$.
First, we consider the case $s=1$.
By Proposition \ref{E-Contribution-Square-001} and \ref{E-C-non-S-001}, we have
\begin{align}
\sum_{P\in\mb P_{g+1}}\sum_{u\in\mc F_{P}} L(1,\chi_{u})  = \tilde{H}(g) \frac{|P|^{2}}{\log_{q}|P|} + O\left(|P|^{\frac{3}{2}}\right),
\end{align}
where
\begin{align}
\tilde{H}(g) := \tilde{A}_{g}(1) - \zeta_{\A}(2) q^{[\frac{g}2]-(g+1)} + \tilde{B}(g) |P|^{-1}.
\end{align}
It is easy to show that
\begin{align}
\tilde{H}(g) \frac{|P|^{2}}{\log_{q}|P|} = \zeta_{\A}(2) \frac{|P|^{2}}{\log_{q}|P|} + O\left(|P|^{\frac{3}{2}}\right).
\end{align}
Hence, we have
\begin{align}
\sum_{P\in\mb P_{g+1}}\sum_{u\in\mc F_{P}} L(1,\chi_{u})  = \zeta_{\A}(2) \frac{|P|^{2}}{\log_{q}|P|} + O\left(|P|^{\frac{3}{2}}\right).
\end{align}
Now, consider the case $s\ne 1$.
For $|s-1|>\varepsilon$, $\eta(s)$ is bounded.
By Proposition \ref{E-Contribution-Square-001} and \ref{E-C-non-S-001}, we have
\begin{align}\label{E-F-M-Real-010}
\sum_{P\in\mb P_{g+1}}\sum_{u\in\mc F_{P}} L(s,\chi_{u}) = \tilde{H}_{g}(s) \frac{|P|^{2}}{\log_{q}|P|} + O\left(|P|^{\frac{3}{2}}\right),
\end{align}
where
\begin{align}
\tilde{H}_{g}(s) := \tilde{A}_{g}(s) - \zeta_{\A}(2) q^{[\frac{g}2]-(g+1)s} + \eta(s) \left\{\tilde{B}_{g}(s) - \zeta_{\A}(2) q^{[\frac{g-1}2]-gs}\right\}.
\end{align}
For $s\in\mb C$ with $1 \le {\rm Re}(s)$ and $|s-1|>\varepsilon$, we have
\begin{align}\label{E-F-M-Real-012}
\zeta_{\A}(2) q^{[\frac{g}2]-(g+1)s} \frac{|P|^{2}}{\log_{q}|P|} &\ll |P|^{\frac{3}{2}},  \\
\eta(s) \left\{\tilde{B}_{g}(s) - \zeta_{\A}(2) q^{[\frac{g-1}2]-gs}\right\} \frac{|P|^{2}}{\log_{q}|P|} &\ll |P|^{\frac{3}{2}}.
\end{align}
Then, by \eqref{E-F-M-Real-012}, we have
\begin{align}\label{E-F-M-Real-013}
\tilde{H}_{g}(s) \frac{|P|^{2}}{\log_{q}|P|} = \tilde{J}_{g}(s) \frac{|P|^{2}}{\log_{q}|P|} + O\left(|P|^{\frac{3}{2}}\right).
\end{align}
Hence, by \eqref{E-F-M-Real-010} and \eqref{E-F-M-Real-013}, we have
\begin{align}
\sum_{P\in\mb P_{g+1}}\sum_{u\in\mc F_{P}} L(s,\chi_{u}) = \tilde{J}_{g}(s) \frac{|P|^{2}}{\log_{q}|P|} + O\left(|P|^{\frac{3}{2}}\right)
\end{align}
for $s\in\mb C$ with $\frac{1}2 \le {\rm Re}(s)$ and $|s-1|>\varepsilon$.
This completes the proof of Theorem \ref{E-F-M-Theorem} (2). \hfill\fbox

\subsubsection{\bf Proof of Theorem \ref{E-F-M-Theorem} (3)}\label{Even-F-M-SS5}
For any $u = v+\xi \in \mc F'_{g+1}$ with $v\in\mc F_{g+1}$, we have $\chi_{u}(f) = (-1)^{\deg(f)} \chi_{v}(f)$.
Then, by Lemma \ref{A-FE} (3), we have
\begin{align}
\sum_{P\in\mb P_{g+1}}\sum_{u\in\mc F'_{P}} L(s,\chi_{u})
&= \sum_{P\in\mb P_{g+1}}\sum_{u\in\mc F_{P}} \sum_{n=0}^{g} (-1)^{n} q^{-sn} \sum_{f\in\A_{n}^{+}} \chi_{u}(f) \nonumber \\
&\quad + (-1)^{g} q^{-(g+1)s} \sum_{P\in\mb P_{g+1}}\sum_{u\in\mc F_{P}} \sum_{n=0}^{g} \sum_{f\in\A_{n}^{+}} \chi_{u}(f) \nonumber \\
&\quad + \nu(s) q^{(1-2s)g} \sum_{P\in\mb P_{g+1}}\sum_{u\in\mc F_{P}} \sum_{n=0}^{g-1} (-1)^{n} q^{(s-1)n} \sum_{f\in\A_{n}^{+}} \chi_{u}(f)  \nonumber \\
&\quad + (-1)^{g+1} \nu(s) q^{-sg} \sum_{P\in\mb P_{g+1}}\sum_{u\in\mc F_{P}} \sum_{n=0}^{g-1} \sum_{f\in\A_{n}^{+}} \chi_{u}(f),
\end{align}
where $\nu(s) := \frac{1+q^{-s}}{1+q^{s-1}}$.
Following the process as in the proof of Theorem \ref{E-F-M-Theorem} (2), we can show that
\begin{align}
\sum_{P\in\mb P_{g+1}}\sum_{u\in\mc F'_{P}} L(s,\chi_{u}) = \tilde{K}_{g}(s) \frac{|P|^{2}}{\log_{q}|P|} + O\left(|P|^{\frac{3}{2}}\right)
\end{align}
for $s\in\mb C$ with $\frac{1}2 \le {\rm Re}(s)$.
This completes the proof of Theorem \ref{E-F-M-Theorem} (3). \hfill\fbox

\section{Second moment of prime $L$-functions at $s=\frac{1}2$}\label{Sect-5}

\subsection{Some lemmas on divisor function}
We present now a few lemmas about the divisor function in $\mathbb{F}_{q}[T]$.

\begin{lem}\label{DIVISOR}
For any positive integer $n$, we have
\begin{equation}\label{DIVISOR1}
\sum_{f\in\A_{n}^{+}}d(f) = n q^{n} + O(q^{n})
\end{equation}
and
\begin{equation}\label{DIVISOR2}
\sum_{f\in\A_{n}^{+}}d(f^{2}) = \frac{1}{2 \zeta_{\A}(2)} n^2 q^{n} + O(n q^{n}).
\end{equation}
\end{lem}
\begin{proof}
Equations \eqref{DIVISOR1} and \eqref{DIVISOR2} are quoted from \cite[Proposition 2.5]{Ro02} and \cite[Lemma 4.4]{AK13}, respectively.
\end{proof}

\begin{rem}\label{REM-DIVISOR}
Let $\rho(f) := d(f^2)$, which is a multiplicative function on $\A^{+}$, and $\zeta_{\rho}(s)$ be the Dirichlet series associated to $\rho$.
In the proof of \cite[Lemma 4.4]{AK13}, it is shown that
\begin{align}
\zeta_{\rho}(s) = \frac{\zeta_{\A}(s)^3}{\zeta_{\A}(2s)} = \frac{1-q^{1-2s}}{(1-q^{1-s})^3}.
\end{align}
Putting $z = q^{-s}$ and considering the power series expansion of $\frac{1-qz^2}{(1-qz)^3}$ at $z=0$, we can see that
\begin{align}\label{DIVISOR3}
\sum_{f\in\A_{n}^{+}} d(f^2) = \left\{1+\tfrac{1}2 (3+q^{-1}) n +\tfrac{1}2 (1-q^{-1})n^2\right\} q^{n}.
\end{align}
\end{rem}

\begin{lem}\label{E-D-SQUARE-2}
Let $P\in\mb P_{r}$. Then, for any integer $n\ge 0$, the value
\begin{align}
\sum_{f\in\A_{n}^{+}, P\nmid f} d(f^2)
\end{align}
is independent of $P$, and depends only on $r$.
Denote this value by $\rho^{*}_{n}(r)$, and let $\rho_{n} = \sum_{f\in\A_{n}^{+}} d(f^2)$.
Then we have
\begin{align}
\rho^{*}_{n}(r) = \begin{cases}
\rho_{n} & \text{ for $0 \le n \le r-1$,} \\
\rho_{n} - 3 \rho_{n-r} & \text{ for $r \le n \le 2r-1$,} \\
\rho_{n}-3\rho_{n-r} + 4 \sum_{l=2}^{m} (-1)^{l} \rho_{n-lr}
& \text{ for $mr \le n < (m+1)r-1$ with $m\ge 2$.}
\end{cases}
\end{align}
\end{lem}
\begin{proof}
Let $\zeta_{\rho}(s)$ be the Dirichlet series given in Remark \ref{REM-DIVISOR}.
Write
\begin{align}\label{E-D-SQUARE-2-001}
\zeta_{\rho}(s) = \sum_{n=0}^{\infty} \rho_{n} q^{-sn} ~~~\text{ with }~~ \rho_{n} = \sum_{f\in\A_{n}^{+}} d(f^2).
\end{align}
Let $\zeta_{\rho}^{*}(s)$ be the power series defined by
\begin{align}
\zeta_{\rho}^{*}(s) = \sum_{\substack{f\in\A^{+}\\ P\nmid f}} \rho(f) |f|^{-s} = \sum_{n=0}^{\infty} \rho^{*}_{n} q^{-sn},
~~~\text{ where }~~ \rho^{*}_{n} = \sum_{\substack{f\in\A^{+}_{n}\\ P\nmid f}} \rho(f).
\end{align}
Then, we have
\begin{align}\label{E-D-SQUARE-2-002}
\zeta_{\rho}(s) = \zeta_{\rho}^{*}(s) \left(1+\sum_{n=1}^{\infty} \rho(P^{n}) |P|^{-ns}\right).
\end{align}
Putting $z=q^{-s}$, since $\rho(P^n) = d(P^{2n}) = 2n+1$, we have
\begin{align}\label{E-D-SQUARE-2-003}
1+\sum_{n=1}^{\infty} \rho(P^{n}) |P|^{-ns} &= 1+\sum_{n=1}^{\infty} (2n+1) |P|^{-ns}
= 1+\sum_{n=1}^{\infty} (2n+1) z^{r n} = \frac{1+z^{r}}{(1-z^{r})^2}.
\end{align}
Hence, by \eqref{E-D-SQUARE-2-001}, \eqref{E-D-SQUARE-2-002} and \eqref{E-D-SQUARE-2-003}, we have
\begin{align}
(1+z^{r})\sum_{n=0}^{\infty} \rho_{n}^{*} z^{n} = (1-z^{r})^2 \sum_{n=0}^{\infty} \rho_{n} z^{n}.
\end{align}
By comparing the coefficients, we have
\begin{align}\label{E-D-SQUARE-2-005}
\begin{cases}
\rho_{n}^{*} = \rho_{n} & \text{ for } 0 \le n \le r-1, \\
\rho_{n}^{*} + \rho_{n-r}^{*} = \rho_{n}-2\rho_{n-r} & \text{ for } r \le n \le 2r-1, \\
\rho_{n}^{*} + \rho_{n-r}^{*} = \rho_{n} - 2\rho_{n-r} + \rho_{n-2r} & \text{ for } 2r \le n.
\end{cases}
\end{align}
From \eqref{E-D-SQUARE-2-005}, we can obtain that $\rho_{n}^{*} = \rho_{n}$ for $0 \le n \le r-1$,
$\rho_{n}^{*} = \rho_{n}-3\rho_{n-r}$ for $r \le n \le 2r-1$ and
$\rho^{*}_{n} = \rho_{n}-3\rho_{n-r} + 4 \sum_{l=2}^{m} (-1)^{l} \rho_{n-lr}$ for $mr \le n < (m+1)r-1$ with $m\ge 2$.
We can also see that the $\rho_{n}^{*}$'s are independent of $P$, and depends only on $r$.
\end{proof}

\subsection{Odd characteristic case}
In this section, we give a proof of Theorem \ref{S-M-Theorem}.
In \S\ref{C-S-S-M}, we obtain several results of the contribution of squares and of non-squares,
which will be used to calculate the second moment of $L$-functions at $s=\frac{1}2$
in \S\ref{subsect5-3} and \S\ref{subsect5-4}.

\subsubsection{\bf Preparations for the proof}\label{C-S-S-M}
We first consider the contribution of squares.

\begin{prop}\label{CSSM-001}
Let $h \in \{2g, 2g-1\}$.
We have
\begin{enumerate}
\item
\begin{align}
\sum_{P\in\mb P_{2g+2}} \sum_{n=0}^{h} (\pm 1)^{n} q^{-\frac{n}{2}} \sum_{\substack{f\in \A_{n}^{+}\\ f =\square}}d(f) \chi_{P}(f)
= \frac{1}{48 \zeta_{\A}(2)} |P| (\log_{q}|P|)^2  + O\left(|P| (\log_{q}|P|)\right),
\end{align}
\item
\begin{align}
q^{-(\frac{h+1}2)} \sum_{P\in\mb P_{2g+2}} \sum_{n=0}^{h} \sum_{\substack{f\in \A_{n}^{+}\\ f =\square}} d(f) \chi_{P}(f)
= O\left(|P|(\log_{q}|P|)\right),
\end{align}
\item
\begin{align}
q^{-(\frac{h+1}2)} \sum_{P\in\mb P_{2g+2}} \sum_{n=0}^{h} (h+1-n) \sum_{\substack{f\in \A_{n}^{+}\\ f =\square}}d(f) \chi_{P}(f)
= O\left(|P|(\log_{q}|P|)\right).
\end{align}
\end{enumerate}
\end{prop}
\begin{proof}
(1) By Theorem \ref{thm:pnt}, we have
\begin{align}\label{CSSM-001-1}
\sum_{P\in\mb P_{2g+2}} \sum_{n=0}^{h} (\pm 1)^{n} q^{-\frac{n}{2}} \sum_{\substack{f\in \A_{n}^{+}\\ f =\square}}d(f) \chi_{P}(f)
&= \sum_{l=0}^{[\frac{h}2]} q^{-l}\sum_{L\in \A_{l}^{+}}  d(L^2) \sum_{P\in\mb P_{2g+2}} \chi_{P}(L^2) \nonumber\\
&\hspace{-6em}= \frac{|P|}{\log_{q}|P|} \sum_{l=0}^{[\frac{h}2]} q^{-l} \sum_{L\in \A_{l}^{+}}d(L^2)
+ O\left(|P|^{\frac1{2}}(\log_{q}|P|)^2\right).
\end{align}
By\eqref{DIVISOR2}, we have
\begin{align}\label{CSSM-001-2}
\frac{|P|}{\log_{q}|P|} \sum_{l=0}^{[\frac{h}2]} q^{-l} \sum_{L\in \A_{l}^{+}}d(L^2)
&=\frac{|P|}{2\zeta_{\A}(2)\log_{q}|P|} \sum_{l=0}^{[\frac{h}2]} l^2 + O\left(|P| (\log_{q}|P|)\right) \nonumber\\
&= \frac{g^3 |P|}{6 \zeta_{\A}(2)\log_{q}|P|} + O\left(|P| (\log_{q}|P|)\right) \nonumber\\
&= \frac{1}{48 \zeta_{\A}(2)} |P| (\log_{q}|P|)^2 + O\left(|P| (\log_{q}|P|)\right).
\end{align}
Since $|P|^{\frac1{2}}(\log_{q}|P|)^2 \ll |P| (\log_{q}|P|)$, by inserting \eqref{CSSM-001-2} into \eqref{CSSM-001-1}, we obtain the desired result.

\noindent(2) By Theorem \ref{thm:pnt} and \eqref{DIVISOR2}, we have
\begin{align}
q^{-(h+\frac12)} \sum_{P\in\mb P_{2g+2}} \sum_{n=0}^{h} \sum_{\substack{f\in \A_{n}^{+}\\ f =\square}}d(f) \chi_{P}(f)
&= q^{-(h+\frac12)} \sum_{P\in\mb P_{2g+2}} \sum_{l=0}^{[\frac{h}2]} \sum_{L\in \A_{l}^{+}} \chi_{P}(L^2) d(L^2) \nonumber\\
&\ll q^{-(h+\frac12)} \frac{|P|}{\log_{q}|P|} \sum_{l=0}^{g} l^2 q^{l} \ll |P|(\log_{q}|P|).
\end{align}
Similarly, we can prove (3).
\end{proof}

Now, We consider the contribution of non-squares.

\begin{prop}\label{CNSSM-001}
Let $h \in \{2g, 2g-1\}$.
We have
\begin{enumerate}
\item
\begin{align}
\sum_{P\in\mb P_{2g+2}} \sum_{n=0}^{h} (\pm 1)^{n} q^{-\frac{n}{2}} \sum_{\substack{f\in \A_{n}^{+}\\ f \ne\square}}d(f) \chi_{P}(f)
= O\left(|P| (\log_{q}|P|)\right),
\end{align}
\item
\begin{align}
q^{-(\frac{h+1}2)} \sum_{P\in\mb P_{2g+2}} \sum_{n=0}^{h} \sum_{\substack{f\in \A_{n}^{+}\\ f \ne\square}}d(f) \chi_{P}(f)
= O\left(|P| (\log_{q}|P|)\right),
\end{align}
\item
\begin{align}
q^{-(\frac{h+1}2)} \sum_{P\in\mb P_{2g+2}} \sum_{n=0}^{h} (h+1-n) \sum_{\substack{f\in \A_{n}^{+}\\ f \ne\square}}d(f) \chi_{P}(f)
= O\left(|P| (\log_{q}|P|)\right).
\end{align}
\end{enumerate}
\end{prop}
\begin{proof}
(1) By Proposition \ref{bound} and \eqref{DIVISOR}, we have
\begin{align}
\sum_{P\in\mb P_{2g+2}} \sum_{n=0}^{h} (\pm 1)^{n} q^{-\frac{n}{2}} \sum_{\substack{f\in \A_{n}^{+}\\ f \ne\square}}d(f) \chi_{P}(f)
&\ll \sum_{n=0}^{h} q^{-\frac{n}{2}} \sum_{\substack{f\in \A_{n}^{+}\\ f \ne\square}}d(f) \left|\sum_{P\in\mb P_{2g+2}} \chi_{P}(f)\right|  \\
&\ll \frac{|P|^{\frac1{2}}}{\log_{q}|P|} \sum_{n=0}^{h} n^2 q^{\frac{n}{2}} \ll |P| (\log_{q}|P|).
\end{align}
Similarly, we can prove (2) and (3).
\end{proof}

\subsubsection{\bf Proof of Theorem \ref{S-M-Theorem} (1)}\label{subsect5-3}
By Lemma \ref{A-FES} (2), we can write
\begin{align}\label{A-FES-Real-001}
\sum_{P\in\mb P_{2g+2}} L(\tfrac{1}2,\chi_{P})^{2}
&= \sum_{P\in\mb P_{2g+2}} \sum_{n=0}^{2g} q^{-\frac{n}{2}} \sum_{f\in\A_{n}^{+}} d(f) \chi_{P}(f)
+ \sum_{P\in\mb P_{2g+2}} \sum_{n=0}^{2g-1} q^{-\frac{n}{2}} \sum_{f\in\A_{n}^{+}} d(f) \chi_{P}(f)  \nonumber \\
&\quad- q^{-(g+\frac12)} \sum_{P\in\mb P_{2g+2}} \sum_{n=0}^{2g} \sum_{f\in\A_{n}^{+}} d(f) \chi_{P}(f)
- q^{-g} \sum_{P\in\mb P_{2g+2}} \sum_{n=0}^{2g-1} \sum_{f\in\A_{n}^{+}} d(f) \chi_{P}(f) \nonumber \\
&\quad - \zeta_{\A}(\tfrac{3}2)^{-1} q^{-(g+\frac12)} \sum_{P\in\mb P_{2g+2}} \sum_{n=0}^{2g} (2g+1-n) \sum_{f\in\A_{n}^{+}} d(f) \chi_{P}(f) \nonumber \\
&\quad - \zeta_{\A}(\tfrac{3}2)^{-1} q^{-g} \sum_{P\in\mb P_{2g+2}} \sum_{n=0}^{2g-1} (2g-n) \sum_{f\in\A_{n}^{+}} d(f) \chi_{P}(f).
\end{align}
In the right hand side of \eqref{A-FES-Real-001}, we can write $\sum_{f\in\A_{n}^{+}} d(f) \chi_{P}(f)$ as
\begin{align}
\sum_{f\in\A_{n}^{+}} d(f) \chi_{P}(f)
= \sum_{\substack{f\in\A_{n}^{+}\\f=\square}} d(f) \chi_{P}(f) + \sum_{\substack{f\in\A_{n}^{+}\\f\ne\square}} d(f) \chi_{P}(f).
\end{align}
Then, by Propositions \ref{CSSM-001} and \ref{CNSSM-001}, we have
\begin{align}
\sum_{P\in\mb P_{2g+2}} L(\tfrac{1}2,\chi_{P})^{2} = \frac{1}{24 \zeta_{\A}(2)} |P| (\log_{q}|P|)^2  + O\left(|P| (\log_{q}|P|)\right).
\end{align}
This completes the proof of Theorem \ref{S-M-Theorem} (1). \hfill\fbox

\subsubsection{\bf Proof of Theorem \ref{S-M-Theorem} (2)}\label{subsect5-4}
For any $f\in\A^{+}$, we have $\chi_{\gamma P}(f) = (-1)^{\deg(f)} \chi_{P}(f)$.
By Lemma \ref{A-FES} (3), we have
\begin{align}
\sum_{P\in\mb P_{2g+2}} L(\tfrac{1}2,\chi_{\gamma P})^{2}
&= \sum_{P\in\mb P_{2g+2}}\sum_{n=0}^{2g} (-1)^{n} q^{-\frac{n}{2}} \sum_{f\in\A_{n}^{+}} d(f) \chi_{P}(f)
+ \sum_{P\in\mb P_{2g+2}}\sum_{n=0}^{2g-1} (-1)^{n} q^{-\frac{n}{2}} \sum_{f\in\A_{n}^{+}} d(f) \chi_{P}(f)\nonumber \\
&\hspace{1.5em} + q^{-(g+\frac12)} \sum_{P\in\mb P_{2g+2}}\sum_{n=0}^{2g} \sum_{f\in\A_{n}^{+}} d(f) \chi_{P}(f)
+ q^{-g} \sum_{P\in\mb P_{2g+2}}\sum_{n=0}^{2g-1} \sum_{f\in\A_{n}^{+}} d(f) \chi_{P}(f)\nonumber \\
&\hspace{1.5em} + \frac{\zeta_{\A}(\tfrac{3}2)}{\zeta_{\A}(2)} q^{-(g+\frac12)} \sum_{P\in\mb P_{2g+2}}
\sum_{n=0}^{2g} (2g+1-n) \sum_{f\in\A_{n}^{+}} d(f) \chi_{P}(f)\nonumber  \\
&\hspace{1.5em} + \frac{\zeta_{\A}(\tfrac{3}2)}{\zeta_{\A}(2)} q^{-g}\sum_{P\in\mb P_{2g+2}}
\sum_{n=0}^{2g-1} (2g-n) \sum_{f\in\A_{n}^{+}} d(f) \chi_{P}(f).
\end{align}
Then, by Propositions \eqref{CSSM-001} and \eqref{CNSSM-001}, we have
\begin{align}
\sum_{P\in\mb P_{2g+2}} L(\tfrac{1}2,\chi_{\gamma P})^{2} = \frac{1}{24 \zeta_{\A}(2)} |P| (\log_{q}|P|)^2  + O\left(|P| (\log_{q}|P|)\right).
\end{align}
This completes the proof of Theorem \ref{S-M-Theorem} (2). \hfill\fbox

\subsection{Even characteristic case}
In this section, we give a proof of Theorem \ref{E-S-M-Theorem}.
In \S\ref{E-C-S-S-M}, we obtain several results of the contribution of squares and of non-squares,
which will be used to calculate the second moment of $L$-functions at $s=\frac{1}2$
in \S\ref{subsect5-3-3}, \S\ref{subsect5-3-4} and \S\ref{subsect5-3-5}.

\subsubsection{\bf Preparations for the proof}\label{E-C-S-S-M}

\begin{prop}\label{E-C-S-S-M-000}
Let $h\in\{2g-1, 2g\}$.
We have
\begin{align}
\sum_{r=1}^{g} \sum_{u\in\mc H_{(r,g+1-r)}} \sum_{n=0}^{h} q^{-\frac{n}2} \sum_{\substack{f\in\A_{n}^{+}\\ f=\square}} d(f) \chi_{u}(f)
= \frac{1}{3 \zeta_{\A}(2)} g^2 q^{2g+1} + O\left(g q^{2g}\right).
\end{align}
\end{prop}
\begin{proof}
We only prove the case $h=2g$.
By similar method, we can prove the case $h=2g-1$.
By Lemma \ref{E-D-SQUARE-2}, we have
\begin{align}
\sum_{r=1}^{g} \sum_{u\in\mc H_{(r,g+1-r)}} \sum_{n=0}^{2g} q^{-\frac{n}2}
\sum_{\substack{f\in\A_{n}^{+}\\ f=\square}} d(f) \chi_{u}(f)
&= \sum_{r=1}^{g} \sum_{u\in\mc H_{(r,g+1-r)}} \sum_{l=0}^{g} q^{-l}
\sum_{L\in\A_{l}^{+}} d(L^{2}) \chi_{u}(L^2)  \nonumber\\
&\hspace{1em}= \sum_{r=1}^{g} \sum_{u\in\mc H_{(r,g+1-r)}} \sum_{l=0}^{g} q^{-l} \rho_{l}^{*}(r).
\end{align}
By Lemma \ref{E-D-SQUARE-2} again, we have
\begin{align}\label{E-C-S-S-M-000-001}
\sum_{r=1}^{g} \sum_{u\in\mc H_{(r,g+1-r)}} \sum_{l=0}^{g} q^{-l} \rho_{l}^{*}(r)
&= \sum_{r=1}^{g} \sum_{u\in\mc H_{(r,g+1-r)}} \sum_{l=0}^{g} q^{-l} \rho_{l}
-3 \sum_{r=1}^{g} \sum_{u\in\mc H_{(r,g+1-r)}} \sum_{l=r}^{g} q^{-l} \rho_{l-r} \nonumber \\
&\hspace{1em} + 4 \sum_{r=1}^{g} \sum_{u\in\mc H_{(r,g+1-r)}} \sum_{m=2}^{[\frac{g}r]}(-1)^{m}\sum_{l=mr}^{g} q^{-l} \rho_{l-mr}.
\end{align}
From \eqref{DIVISOR3}, we have $q^{-n}\rho_{n} = 1+\tfrac{1}2 (3+q^{-1}) n +\tfrac{1}2 (1-q^{-1}) n^2$.
Hence, we have
\begin{align}
\sum_{l=0}^{g} q^{-l} \rho_{l} &= \tfrac{1}{6} (1-q^{-1}) g^3 + O(g^2), \\
\sum_{l=r}^{g} q^{-l} \rho_{l-r} &= q^{-r} \sum_{k=0}^{g-r} q^{-k} \rho_{k} \ll q^{-r} g^3, \\
\sum_{m=2}^{[\frac{g}r]}(-1)^{m}\sum_{l=mr}^{g} q^{-l} \rho_{l-mr} &= \sum_{m=2}^{[\frac{g}r]}(-1)^{m} q^{-mr} \sum_{k=0}^{g-mr} q^{-k} \rho_{k}
\ll g^3 \sum_{m=2}^{[\frac{g}r]} q^{-mr} \ll \frac{g^4}{r} q^{-2r}.
\end{align}
Then, using \eqref{E-PNT-2-2} and the fact that $\#\mc H_{(r,g+1-r)} \le q^{g+1} \frac{q^{r}}{r}$, we have
\begin{align}
\sum_{r=1}^{g} \sum_{u\in\mc H_{(r,g+1-r)}} \sum_{l=0}^{g} q^{-l} \rho_{l}
&= \frac{1}{3} (1-q^{-1}) g^2 q^{2g+1} + O\left(g q^{2g}\right), \label{E-C-S-S-M-000-002}\\
\sum_{r=1}^{g} \sum_{u\in\mc H_{(r,g+1-r)}} \sum_{l=r}^{g} q^{-l} \rho_{l-r}
&\ll g^3 q^{g} \sum_{r=1}^{g} \frac{1}{r} \ll (\log g) g^3 q^{g} \ll g q^{2g}, \label{E-C-S-S-M-000-003}\\
\sum_{r=1}^{g} \sum_{u\in\mc H_{(r,g+1-r)}} \sum_{m=2}^{[\frac{g}r]}(-1)^{m}\sum_{l=mr}^{g} q^{-l} \rho_{l-mr}
&\ll g^4 q^{g} \sum_{r=1}^{g} \frac{q^{-r}}{r^2} \ll g^5 q^g \ll g q^{2g}. \label{E-C-S-S-M-000-004}
\end{align}
By inserting \eqref{E-C-S-S-M-000-002}, \eqref{E-C-S-S-M-000-003} and \eqref{E-C-S-S-M-000-004} into \eqref{E-C-S-S-M-000-001}, we get the result.
\end{proof}

\begin{prop}\label{E-C-S-S-M-001}
Let $h\in\{2g-1, 2g\}$.
We have
\begin{enumerate}
\item
\begin{align}
\sum_{P\in\mb P_{g+1}}\sum_{u\in\mc F_{P}}\sum_{n=0}^{h} (\pm 1)^{n} q^{-\frac{n}{2}}\sum_{\substack{f\in\A_{n}^+\\ f=\square}}d(f) \chi_{u}(f)
= \frac{1}{6\zeta_{\A}(2)} |P|^{2} (\log_{q}|P|)^2 + O\left(|P|^{2} (\log_{q}|P|)\right),
\end{align}
\item
\begin{align}
q^{-(\frac{h+1}2)} \sum_{P\in\mb P_{g+1}}\sum_{u\in\mc F_{P}} \sum_{n=0}^{h} \sum_{\substack{f\in \A_{n}^{+}\\ f =\square}}d(f) \chi_{u}(f)
= O\left(|P|^{2} (\log_{q}|P|)\right),
\end{align}
\item
\begin{align}
q^{-(\frac{h+1}2)} \sum_{P\in\mb P_{g+1}}\sum_{u\in\mc F_{P}} \sum_{n=0}^{h} (h+1-n) \sum_{\substack{f\in \A_{n}^{+}\\ f =\square}}d(f) \chi_{u}(f)
= O\left(|P|^{2} (\log_{q}|P|)\right).
\end{align}
\end{enumerate}
\end{prop}
\begin{proof}
(1) By \eqref{E-PNT-2-1}, we have
\begin{align}\label{E-C-S-S-M-001-1}
\sum_{P\in\mb P_{g+1}}\sum_{u\in\mc F_{P}}\sum_{n=0}^{h} (\pm 1)^{n} q^{-\frac{n}{2}}\sum_{\substack{f\in\A_{n}^+\\ f=\square}}d(f) \chi_{u}(f)
&= \sum_{l=0}^{[\frac{h}2]} q^{-l}\sum_{L\in \A_{l}^{+}} d(L^2) \sum_{P\in\mb P_{g+1}}\sum_{u\in\mc F_{P}}\chi_{u}(L^2) \nonumber \\
&\hspace{-8em}= \frac{|P|^{2}}{\log_{q}|P|} \sum_{l=0}^{[\frac{h}2]} q^{-l} \sum_{L\in \A_{l}^{+}}d(L^2) + O\left(|P|^{\frac{3}{2}} (\log_{q}|P|)^2\right).
\end{align}
By \eqref{DIVISOR2}, we have
\begin{align}\label{E-C-S-S-M-001-2}
\frac{|P|^{2}}{\log_{q}|P|} \sum_{l=0}^{[\frac{h}2]} q^{-l} \sum_{L\in \A_{l}^{+}}d(L^2)
&= \frac{1}{2\zeta_{\A}(2)} \frac{|P|^{2}}{\log_{q}|P|} \sum_{l=0}^{[\frac{h}2]} l^2 + O\left(|P|^{2} (\log_{q}|P|)\right)  \nonumber \\
&= \frac{1}{6\zeta_{\A}(2)} |P|^{2} (\log_{q}|P|)^2 + O\left(|P|^{2} (\log_{q}|P|)\right).
\end{align}
Since $|P|^{\frac{3}{2}} (\log_{q}|P|)^2 \ll |P|^{2} (\log_{q}|P|)$, by inserting \eqref{E-C-S-S-M-001-2} into \eqref{E-C-S-S-M-001-1}, we get the result.

\noindent(2) By \eqref{E-PNT-2-1} and \eqref{DIVISOR2}, we have
\begin{align}
q^{-(\frac{h+1}2)} \sum_{P\in\mb P_{g+1}}\sum_{u\in\mc F_{P}} \sum_{n=0}^{h} \sum_{\substack{f\in \A_{n}^{+}\\ f =\square}}d(f) \chi_{u}(f)
&= q^{-(\frac{h+1}2)} \sum_{P\in\mb P_{g+1}}\sum_{u\in\mc F_{P}} \sum_{l=0}^{[\frac{h}2]} \sum_{L\in \A_{l}^{+}} d(L^2) \chi_{u}(L^2)  \nonumber\\
&\ll q^{-(\frac{h+1}2)} \frac{|P|^{2}}{\log_{q}|P|} \sum_{l=0}^{[\frac{h}2]} l^2 q^{l} \ll |P|^{2} (\log_{q}|P|).
\end{align}
Similarly, we can prove (3).
\end{proof}

Now, we consider the contribution of non-squares.

\begin{prop}\label{E-C-NS-S-M-000}
Let $h\in\{2g-1, 2g\}$.
We have
\begin{align}
\sum_{u\in\mc H_{g+1}} \sum_{n=0}^{h} q^{-\frac{n}2} \sum_{\substack{f\in\A_{n}^{+}\\ f\ne\square}} d(f) \chi_{u}(f)
= O\left((\log g) g q^{2g}\right).
\end{align}
\end{prop}
\begin{proof}
We only prove the case $h=2g$.
Similarly, we can prove the case $h=2g-1$.
By Lemma \ref{E-BOUND-3-1} and \eqref{DIVISOR1}, we have
\begin{align}
&\sum_{u\in\mc H_{g+1}} \sum_{n=0}^{2g} q^{-\frac{n}2} \sum_{\substack{f\in\A_{n}^{+}\\ f\ne\square}} d(f) \chi_{u}(f)
\ll \sum_{n=0}^{2g} q^{-\frac{n}2} \sum_{\substack{f\in\A_{n}^{+} \\ f \ne \square}} d(f)
\left|\sum_{u\in\mc H_{g+1}} \chi_{u}(f)\right|  \\
&\hspace{2em}\ll (\log g) q^{g} \sum_{n=0}^{2g} q^{-\frac{n}2} \sum_{f\in\A_{n}^{+}} d(f)
\ll (\log g) q^{g} \sum_{n=0}^{2g} n q^{\frac{n}2} \ll (\log g) g q^{2g}.
\end{align}
\end{proof}

\begin{prop}\label{E-C-NS-S-M-001}
Let $h\in\{2g-1, 2g\}$.
We have
\begin{enumerate}
\item
\begin{align}
\sum_{P\in\mb P_{g+1}}\sum_{u\in\mc F_{P}}\sum_{n=0}^{h} (\pm 1)^{n} q^{-\frac{n}{2}}\sum_{\substack{f\in\A_{n}^+\\ f\ne\square}}d(f) \chi_{u}(f)
= O\left(|P|^{2}\right),
\end{align}
\item
\begin{align}
q^{-(\frac{h+1}2)} \sum_{P\in\mb P_{g+1}}\sum_{u\in\mc F_{P}} \sum_{n=0}^{h} \sum_{\substack{f\in \A_{n}^{+}\\ f \ne\square}}d(f) \chi_{u}(f)
= O\left(|P|^{2}\right),
\end{align}
\item
\begin{align}
q^{-(\frac{h+1}2)} \sum_{P\in\mb P_{g+1}}\sum_{u\in\mc F_{P}} \sum_{n=0}^{h} (h+1-n) \sum_{\substack{f\in \A_{n}^{+}\\ f \ne\square}}d(f) \chi_{u}(f)
= O\left(|P|^{2}\right).
\end{align}
\end{enumerate}
\end{prop}
\begin{proof}
(1) By Lemma \ref{E-BOUND-2-1} and \eqref{DIVISOR1}, we have
\begin{align}
\sum_{P\in\mb P_{g+1}}\sum_{u\in\mc F_{P}}\sum_{n=0}^{h} (\pm 1)^{n} q^{-\frac{n}{2}}\sum_{\substack{f\in\A_{n}^+\\ f\ne\square}}d(f) \chi_{u}(f)
&\ll \sum_{n=0}^{h} q^{-\frac{n}{2}}\sum_{\substack{f\in\A_{n}^+\\ f\ne\square}}d(f) \left|\sum_{P\in\mb P_{g+1}}\sum_{u\in\mc F_{P}} \chi_{u}(f)\right| \\
&\ll \frac{|P|}{\log_{q}|P|} \sum_{n=0}^{h} n q^{\frac{n}{2}}  \ll |P|^{2}.
\end{align}
Similarly, we can prove (2) and (3).
\end{proof}

\subsubsection{\bf Proof of Theorem \ref{E-S-M-Theorem} (1)}\label{subsect5-3-3}
By Lemma \ref{A-FES} (1), we have
\begin{align}
\sum_{u\in\mc H_{g+1}} L(\tfrac{1}2,\chi_{u})^{2} = \sum_{u\in\mc H_{g+1}} \sum_{n=0}^{2g} q^{-\frac{n}2} \sum_{f\in\A_{n}^{+}} d(f) \chi_{u}(f)
+ \sum_{u\in\mc H_{g+1}} \sum_{n=0}^{2g-1} q^{-\frac{n}2} \sum_{f\in\A_{n}^{+}} d(f) \chi_{u}(f).
\end{align}
Then, by Propositions \ref{E-C-S-S-M-000} and \ref{E-C-NS-S-M-000}, we have
\begin{align}
\sum_{u\in\mc H_{g+1}} L(\tfrac{1}2,\chi_{u})^{2}
= \frac{2}{3 \zeta_{\A}(2)} g^2 q^{2g+1} + O\left((\log g) g q^{2g}\right).
\end{align}
This completes the proof of Theorem \ref{E-S-M-Theorem} (1). \hfill\fbox

\subsubsection{\bf Proof of Theorem \ref{E-S-M-Theorem} (2)}\label{subsect5-3-4}
By Lemma \ref{A-FES} (2), we have
\begin{align}
\sum_{P\in\mb P_{g+1}}\sum_{u\in\mc F_{P}} L(\tfrac{1}2,\chi_{u})^{2}
&= \sum_{P\in\mb P_{g+1}}\sum_{u\in\mc F_{P}} \sum_{n=0}^{2g} q^{-\frac{n}{2}}\sum_{f\in\A_{n}^{+}} d(f) \chi_{u}(f) \nonumber \\
&\quad + \sum_{P\in\mb P_{g+1}}\sum_{u\in\mc F_{P}}\sum_{n=0}^{2g-1} q^{-\frac{n}{2}} \sum_{f\in\A_{n}^{+}} d(f) \chi_{u}(f) \nonumber \\
&\quad - q^{-(g+\frac12)} \sum_{P\in\mb P_{g+1}}\sum_{u\in\mc F_{P}}\sum_{n=0}^{2g} \sum_{f\in\A_{n}^{+}} d(f) \chi_{u}(f) \nonumber \\
&\quad - q^{-g} \sum_{n=0}^{2g-1} \sum_{P\in\mb P_{g+1}}\sum_{u\in\mc F_{P}} \sum_{f\in\A_{n}^{+}} d(f) \chi_{u}(f) \nonumber \\
&\quad - \zeta_{\A}(\tfrac{3}2)^{-1} q^{-(g+\frac12)} \sum_{P\in\mb P_{g+1}}\sum_{u\in\mc F_{P}}
\sum_{n=0}^{2g} (2g+1-n) \sum_{f\in\A_{n}^{+}} d(f) \chi_{u}(f) \nonumber \\
&\quad - \zeta_{\A}(\tfrac{3}2)^{-1} q^{-g}\sum_{P\in\mb P_{g+1}}\sum_{u\in\mc F_{P}}\sum_{n=0}^{2g-1} (2g-n) \sum_{f\in\A_{n}^{+}} d(f) \chi_{u}(f).
\end{align}
Then, by Propositions \ref{E-C-S-S-M-001} and \ref{E-C-NS-S-M-001}, we have
\begin{align}
\sum_{P\in\mb P_{g+1}}\sum_{u\in\mc F_{P}} L(\tfrac{1}2,\chi_{u})^{2}  = \frac{1}{3\zeta_{\A}(2)} |P|^{2} (\log_{q}|P|)^2 + O\left(|P|^{2} (\log_{q}|P|)\right).
\end{align}
This completes the proof of Theorem \ref{E-S-M-Theorem} (2). \hfill\fbox

\subsubsection{\bf Proof of Theorem \ref{E-S-M-Theorem} (3)}\label{subsect5-3-5}
For any $u = v+\xi \in \mc F'_{g+1}$ with $v\in\mc F_{g+1}$, we have $\chi_{u}(f) = (-1)^{\deg(f)} \chi_{v}(f)$.
Then, by Lemma \ref{A-FES} (3), we have
\begin{align}
\sum_{P\in\mb P_{g+1}}\sum_{u\in\mc F'_{P}} L(\tfrac{1}2,\chi_{u})^{2}
&= \sum_{P\in\mb P_{g+1}}\sum_{u\in\mc F_{P}}\sum_{n=0}^{2g} (-1)^n q^{-\frac{n}{2}}\sum_{f\in\A_{n}^{+}} d(f) \chi_{u}(f) \nonumber\\
&\quad + \sum_{P\in\mb P_{g+1}}\sum_{u\in\mc F_{P}}\sum_{n=0}^{2g-1} (-1)^n q^{-\frac{n}{2}} \sum_{f\in\A_{n}^{+}} d(f) \chi_{u}(f) \nonumber\\
&\quad + q^{-(g+\frac12)} \sum_{P\in\mb P_{g+1}}\sum_{u\in\mc F_{P}}\sum_{n=0}^{2g} \sum_{f\in\A_{n}^{+}} d(f) \chi_{u}(f) \nonumber\\
&\quad + q^{-g} \sum_{P\in\mb P_{g+1}}\sum_{u\in\mc F_{P}}\sum_{n=0}^{2g-1} \sum_{f\in\A_{n}^{+}} d(f) \chi_{u}(f) \nonumber\\
&\quad + \frac{\zeta_{\A}(\tfrac{3}2)}{\zeta_{\A}(2)} q^{-(g+\frac12)} \sum_{P\in\mb P_{g+1}}\sum_{u\in\mc F_{P}}
\sum_{n=0}^{2g} (2g+1-n) \sum_{f\in\A_{n}^{+}} d(f) \chi_{u}(f)  \nonumber \\
&\quad + \frac{\zeta_{\A}(\tfrac{3}2)}{\zeta_{\A}(2)} q^{-g}\sum_{P\in\mb P_{g+1}}\sum_{u\in\mc F_{P}}
\sum_{n=0}^{2g-1} (2g-n) \sum_{f\in\A_{n}^{+}} d(f) \chi_{u}(f).
\end{align}
By Propositions \ref{E-C-S-S-M-001} and \ref{E-C-NS-S-M-001}, we have
\begin{align}
\sum_{P\in\mb P_{g+1}}\sum_{u\in\mc F'_{P}} L(\tfrac{1}2,\chi_{u})^{2}
= \frac{1}{3\zeta_{\A}(2)} |P|^{2} (\log_{q}|P|)^2 + O\left(|P|^{2} (\log_{q}|P|)\right).
\end{align}
This completes the proof of Theorem \ref{E-S-M-Theorem} (3). \hfill\fbox

\subsection*{Acknowledgments}
The first author was supported by an EPSRC-IH\'{E}S William Hodge Fellowship and by the EPSRC grant EP/K021132X/1.
The second and third authors were supported by the National Research Foundation of Korea(NRF) grant funded by the Korea government(MSIP)(No. 2014001824).

We also would like to thank two anonymous referees for the careful reading of the paper and the detailed comments that helped to improve the presentation and the clarity of the present work.



\begin{bibdiv}
\begin{biblist}

%

%

\bib{AK1}{article}{
   author={Andrade, J. C.},
   title={A note on the mean value of $L$-functions in function fields},
   journal={Int. J. Number Theory},
   volume={08},
   date={2012},
   number={07},
   pages={1725--1740},
}

\bib{AK2}{article}{
   author={Andrade, J. C.},
   author={Keating, J. P.},
   title={The mean value of $L(\tfrac{1}{2},\chi)$ in the hyperelliptic ensemble},
   journal={J. Number Theory},
   volume={132},
   date={2012},
   number={12},
   pages={2793--2816},
}

\bib{AK3}{article}{
   author={Andrade, J. C.},
   author={Keating, J. P.},
   title={Conjectures for the integral moments and ratios of $L$-functions over function fields},
   journal={J. Number Theory},
   volume={142},
   date={2014},
   number={},
   pages={102--148},
}

\bib{AK13}{article}{
   author={Andrade, J. C.},
   author={Keating, J. P.},
   title={Mean value theorems for $L$-functions over prime polynomials for the rational function field},
   journal={Acta. Arith.},
   volume={161},
   date={2013},
   number={4},
   pages={371--385},
}


\bib{Art}{article}{
author={Artin, E.}
title={Quadratische K\"{o}rper in Geibiet der H\"{o}heren Kongruzzen I and II}
journal={Math. Z.}
volume={19},
date={1924},
number={}
pages={153--296},
}

\bib{Ch08}{article}{
   author={Chen, Yen-Mei J.},
   title={Average values of $L$-functions in characteristic two},
   journal={J. Number Theory},
   volume={128},
   date={2008},
   number={7},
   pages={2138--2158},
}

\bib{CY08}{article}{
   author={Chen, Yen-Mei J.},
   author={Yu, Jing},
   title={On class number relations in characteristic two},
   journal={Math. Z.},
   volume={259},
   date={2008},
   number={1},
   pages={197--216},
}

\bib{Chow}{book}{
author={Chowla, S. D.}
title={The Riemann Hypothesis and Hilbert's Tenth Problem}
publisher={Gordon and Breach Science Publishers}
place={New York}
date={1965}
pages={}
isbn={}
review={}
}

\bib{Ell}{article}{
author={Elliot, P. D. T. A.},
title={On the Distribution of the Values of Quadratic $L$--Series in the Half--Plane $\sigma>\tfrac{1}{2}$},
journal={Invent. Math.},
volume={21},
date={1973},
number={}
pages={319--338}

}

\bib{GV}{article}{
author={Goldfeld, D.},
author={Viola, C.},
title={Mean Values of $L$--Functions Associated to Elliptic, Fermat and Other Curves at the Centre of the Critical Strip},
journal={J. Number Theory}
volume={11}
date={1979}
number={}
pages={305--320}
}

\bib{HR}{article}{
author={Hoffstein, J.},
author={Rosen, M.},
title={Average values of $L$-series in function fields},
journal={J. Reine Angew. Math.}
volume={426}
date={1992}
number={}
pages={117--150}
}

\bib{HL10}{article}{
   author={Hu, Su},
   author={Li, Yan},
   title={The genus fields of Artin-Schreier extensions},
   journal={Finite Fields Appl.},
   volume={16},
   date={2010},
   number={4},
   pages={255--264},
}

%


\bib{Jut}{article}{
author={Jutila, M.},
title={On the Mean Value of $L(\tfrac{1}{2},\chi)$ for Real Characters},
journal={Analysis 1}
volume={}
date={1981}
number={}
pages={149--161}
}

\bib{Po10}{article}{
   author={Pollack, Paul},
   title={Revisiting Gauss's analogue of the prime number theorem for
   polynomials over a finite field},
   journal={Finite Fields Appl.},
   volume={16},
   date={2010},
   number={4},
   pages={290--299},
}

\bib{Ro95}{article} {
    AUTHOR = {Rosen, Michael},
     TITLE = {Average value of {$\vert K_2({\scr O})\vert $} in function fields},
   JOURNAL = {Finite Fields Appl.},
    VOLUME = {1},
      YEAR = {1995},
    NUMBER = {2},
     PAGES = {235--241},
}
		

\bib{Ro02}{book}{
   author={Rosen, Michael},
   title={Number theory in function fields},
   series={Graduate Texts in Mathematics},
   volume={210},
   publisher={Springer-Verlag},
   place={New York},
   date={2002},
   pages={xii+358},
   isbn={0-387-95335-3},
   review={\MR{1876657 (2003d:11171)}},
}

\bib{Ru}{article}{
author={Rudnick, Z.},
title={Traces of high powers of the Frobenius class in the hyperelliptic ensemble},
journal={Acta. Arith.},
voulme={143}
date={2010}
number={}
pages={81--99}
}

\bib{Stan}{article}{
author={Stankus, E.},
title={Distribution of Dirichlet's $L$--functions with real characters in the half--plane $\mathrm{Re} \ s>\tfrac{1}{2}$}
journal={Liet. mat. rink.}
volume={15}
date={1975}
number={}
pages={199--214}
}

\bib{Weil}{book}{
author={Weil, A.},
title={Sur les Courbes Alg\'{e}briques et les Vari\'{e}t\'{e}s qui s'en D\'{e}duisent},
publisher={Hermann},
place={Paris},
date={1948},
pages={},
isbn={},
review={},
}

\end{biblist}
\end{bibdiv}
\end{document}